\documentclass[11pt]{amsart}

\newcommand{\FF}{\mathbb{F}}
\newcommand{\GG}{\mathbb{G}}
\newcommand{\NN}{\mathbb{N}}

\newcommand{\ZZ}{\mathbb{Z}}

\newcommand{\St}{\mathrm{St}}

\newcommand{\cA}{\mathcal{A}}
\newcommand{\cB}{\mathcal{B}}

\newcommand{\cF}{\mathcal{F}}
\newcommand{\cG}{\mathcal{G}}
\newcommand{\cH}{\mathcal{H}}
\newcommand{\cI}{\mathcal{I}}

\newcommand{\cN}{\mathcal{N}}

\newcommand{\cU}{\mathcal{U}}
\newcommand{\cV}{\mathcal{V}}

\newcommand{\fA}{\mathfrak{A}}
\newcommand{\fE}{\mathfrak{E}}

\newcommand{\fF}{\mathfrak{F}}
\newcommand{\fG}{\mathfrak{G}}
\newcommand{\fg}{\mathfrak{g}}

\newcommand{\fM}{\mathfrak{M}}

\newcommand{\fsl}{\mathfrak{sl}}

\newcommand{\fu}{\mathfrak{u}}

\newcommand{\dact}{\boldsymbol{.}}
\newcommand{\lra}{\longrightarrow}

\DeclareMathOperator{\ad}{ad}
\DeclareMathOperator{\Ad}{Ad}
\DeclareMathOperator{\Char}{char}
\DeclareMathOperator{\cx}{cx}
\DeclareMathOperator{\dep}{dp}
\DeclareMathOperator{\Dist}{Dist}
\DeclareMathOperator{\Ext}{Ext}
\DeclareMathOperator{\gr}{gr}

\DeclareMathOperator{\HH}{H}
\DeclareMathOperator{\height}{ht}
\DeclareMathOperator{\Ht}{Ht}
\DeclareMathOperator{\Hom}{Hom}
\DeclareMathOperator{\im}{im}
\DeclareMathOperator{\id}{id}

\DeclareMathOperator{\Lie}{Lie}
\DeclareMathOperator{\modd}{mod}
\DeclareMathOperator{\ph}{ph}
\DeclareMathOperator{\Proj}{Proj}
\DeclareMathOperator{\PSL}{PSL}
\DeclareMathOperator{\Rad}{Rad}

\DeclareMathOperator{\SL}{SL}
\DeclareMathOperator{\Soc}{Soc}
\DeclareMathOperator{\supp}{supp}
\DeclareMathOperator{\Top}{Top}
\DeclareMathOperator{\tr}{tr}
\DeclareMathOperator{\wt}{wt}

\usepackage{pictex}
\usepackage{amscd}
\usepackage{latexsym}
\usepackage{amssymb}
\usepackage{eucal}
\usepackage{eufrak}
\usepackage{verbatim}

\numberwithin{equation}{section}

\newtheorem{Theorem}[equation]{Theorem}
\newtheorem{Lemma}[equation]{Lemma}
\newtheorem{Proposition}[equation]{Proposition}
\newtheorem{Corollary}[equation]{Corollary}

\newtheorem*{thm*}{Theorem}
\theoremstyle{remark}
\newtheorem*{Remark}{Remark}

\newtheorem*{Definition}{Definition}
\newtheorem*{Example}{Example}
\newtheorem*{Examples}{Examples}

\numberwithin{equation}{section}
\theoremstyle{Theorem}
\newtheorem{Thm}{Theorem}[subsection]
\newtheorem{Lem}[Thm]{Lemma}
\newtheorem{Prop}[Thm]{Proposition}
\newtheorem{Cor}[Thm]{Corollary}

\begin{document}

\title{Complexity, Periodicity and One-Parameter Subgroups}

\author[R. Farnsteiner]{Rolf Farnsteiner}

\address{Mathematisches Seminar, Christian-Albrechts-Universit\"at zu Kiel, Ludewig-Meyn-Str. 4, 24098 Kiel, Germany}
\email{rolf@math.uni-kiel.de}
\thanks{Supported by the D.F.G. priority program SPP1388 `Darstellungstheorie'.}
\date{\today}

\makeatletter
%\@addtoreset{subabschnitt}{abschnitt}
\makeatother

\begin{abstract} Using the variety of infinitesimal one-parameter subgroups introduced in \cite{SFB1,SFB2} by Suslin-Friedlander-Bendel, we define a numerical invariant for representations of
an infinitesimal group scheme $\cG$. For an indecomposable $\cG$-module $M$ of complexity $1$, this number, which may also interpreted as the height of a ``vertex" $\cU_M
\subseteq \cG$, is related to the period of $M$. In the context of the Frobenius category of $G_rT$-modules associated to a smooth reductive group $G$ and a maximal torus $T \subseteq G$, 
our methods give control over the behavior of the Heller operator of such modules, as well as precise values for the periodicity of their restrictions to $G_r$. Applications include the structure 
of stable Auslander-Reiten components of $G_rT$-modules as well as the distribution of baby Verma modules. \end{abstract}

\maketitle

\setcounter{section}{-1}

\section{Introduction} \label{S:IP}
This paper is concerned with representations of finite group schemes that are defined over an algebraically closed field $k$ of positive characteristic $p>0$. Given such a group scheme $\cG$
and a finite-dimensional $\cG$-module $M$, Friedlander and Suslin showed in their groundbreaking paper \cite{FS} that the cohomology space $\HH^\ast(\cG,M)$ is a finite module over
the finitely generated commutative $k$-algebra $\HH^\bullet(\cG,k) := \bigoplus_{n\ge 0} \HH^{2n}(\cG,k)$. This result has seen a number of applications which have provided deep
insight into the representation theory of $\cG$.

By work of Alperin-Evens \cite{AE}, Carlson \cite{Ca2} and Suslin-Friedlander-Bendel \cite{SFB1,SFB2}, the notions of complexity, periodicity and infinitesimal one-parameter subgroups are
closely related to properties of the even cohomology ring $\HH^\bullet(\cG,k)$. In this paper, we employ the algebro-geometric techniques expounded in \cite{FS} and \cite{SFB1,SFB2} in
order to obtain information on the period of periodic modules, and the structure of the stable Auslander-Reiten quivers of algebraic groups. Our methods are most effective when covering 
techniques related to $G_rT$-modules can be brought to bear. By way of illustration, we summarize some of our results in the following:

\bigskip

\begin{thm*} Let $G$ be a smooth, reductive group scheme with maximal torus $T \subseteq G$. Suppose that $M$ is an indecomposable $G_rT$-module of complexity $\cx_{G_rT}(M)=1$. 
Then the following statements hold:

{\rm (1)} \ There exists a unipotent subgroup $\cU_M \subseteq G_r$ of height $h_M$ and a root $\alpha$ of $G$ such that $\Omega_{G_rT}^{2p^{r-h_M}}(M) \cong
M\!\otimes_k\!k_{p^r\alpha}$.

{\rm (2)} \ The restriction $M|_{G_r}$ is periodic with period $2p^{r-h_M}$.\end{thm*}

\bigskip
\noindent
Here $\cx_{G_rT}(M)$ refers to the polynomial rate of growth of a minimal projective resolution of $M$ and $\Omega_{G_rT}$ is the Heller operator of the Frobenius category of 
$G_rT$-modules.

Our article can roughly be divided into two parts. Sections \ref{S:UB}-\ref{S:BAR} are mainly concerned with the category $\modd \cG$ of finite-dimensional modules of a finite group scheme
$\cG$. The Frobenius category $\modd G_rT$ of compatibly graded $G_r$-modules is dealt with in the remaining two sections.

In Section \ref{S:UB} we exploit the detailed information provided by the Friedlander-Suslin Theorem \cite{FS} in order to provide an upper bound for the complexity of a $\cG$-module in 
terms of the dimension of certain $\Ext$-groups. Section \ref{S:VG} lays the foundation for the later developments by collecting basic results concerning the cohomology rings of the Frobenius 
kernels $\GG_{a(r)}$ of the additive group $\GG_a$.

The period of a periodic module is known to be closely related to the degrees of homogeneous generators of the ring $\HH^\bullet(\cG,k)$. The Friedlander-Suslin Theorem thus implies that,
for any infinitesimal group $\cG$ of height $\height(\cG)$, the number $2p^{\height(\cG)-1}$ is a multiple of the period of any periodic $\cG$-module. In Section \ref{S:PH}, we analyze
this feature more closely, showing how the period may be bounded by employing infinitesimal one-parameter subgroups of $\cG$. The relevant notion is that of the projective height of a
module, which, for an indecomposable $\cG$-module $M$ of complexity $1$, coincides with the height of a certain unipotent subgroup $\cU_M \subseteq \cG$.

Sections \ref{S:EC} and \ref{S:Rep} are concerned with the Auslander-Reiten theory of finite group schemes. Following a discussion of components of Euclidean tree class, we determine in
Section \ref{S:Rep} those AR-components of the Frobenius kernels $\SL(2)_r$, that contain a simple module. Applications concerning blocks and AR-components of Frobenius kernels of
reductive groups are given in Section \ref{S:BAR}. In particular, we attach to every representation-finite block $\cB \subseteq k\cG$ a unipotent ``defect group" $\cU_\cB \subseteq \cG$,
whose height is linked to the structure of $\cB$.

By providing an explicit formula for the Nakayama functor of the Frobenius category of graded modules over certain Hopf algebras, Section \ref{S:NF} initiates our study of $G_rT$-modules. In 
Section \ref{S:AR} we come to the second central topic of our paper, the Auslander-Reiten theory of the groups $G_r$ and $G_rT$, defined by the r-th Frobenius kernel of a smooth group $G$, 
and a maximal torus $T \subseteq G$. Our first main result, Theorem \ref{MC1}, links the projective height of a $G_rT$-module of complexity $1$ to the behavior of powers of the Heller 
operator $\Omega_{G_rT}$. In particular, the Frobenius category $\modd G_rT$ of a reductive group $G$ is shown to afford no $\Omega_{G_rT}$-periodic modules, and the
$\Omega_{G_r}$-periods of $G_rT$-modules of complexity $1$ are determined by their projective height (cf.\ the Theorem above). It is interesting to compare this fact with the periods of
periodic modules over finite groups, which are given by the minimal ranks of maximal elementary abelian $p$-groups, see \cite[(2.2)]{BC}. The aforementioned results provide insight into the 
structure of the components of the stable Auslander-Reiten quiver of $\modd G_rT$ and the distribution of baby Verma modules:

\bigskip

\begin{thm*} Suppose that $G$ is defined over the Galois field $\FF_p$ with $p\ge 7$. Let $T \subseteq G$ be a maximal torus. 

{\rm (1)} \ If $\Theta$ is a component of the stable Auslander-Reiten quiver of $\modd G_rT$, then $\Theta \cong \ZZ[A_\infty],$ $ \ZZ[A^\infty_\infty],\, \ZZ[D_\infty]$.

{\rm (2)} \ If $\Theta$ contains two baby Verma modules $\widehat{Z}_r(\lambda) \not \cong \widehat{Z}_r(\mu)$, then $\cx_{G_rT}(\widehat{Z}_r(\lambda))=1$, and there exists a simple root $\alpha$ of $G$ such that $\{\widehat{Z}_r(\lambda\!+\!np^r\alpha) \ ; \ n \in \ZZ\}$ is the set of baby Verma modules belonging to $\Theta$.

{\rm (3)} \ A stable Auslander-Reiten component of $\modd G_r$ contains at most one baby Verma module. \end{thm*}

\noindent
By part (2) above, a stable AR-component $\Theta$ of $\modd G_rT$ whose rank variety $V_r(G)_\Theta$ has dimension $\ge 2$ contains at most one baby Verma module, while for $\dim 
V_r(G)_\Theta = 1$ the presence of such a module implies that the baby Verma modules of $\Theta$ form the $\Omega^2_{G_rT}$-orbit of quasi-simple modules.

\bigskip
\noindent
Given a finite group scheme $\cG$ with coordinate ring $k[\cG]$, we let $k\cG :=k[\cG]^\ast$ be its {\it algebra of measures}. By general theory, the representations of this algebra coincide
with those of the group scheme $\cG$, and we denote by $\modd \cG$ the category of finite-dimensional $\cG$-modules. For a $\cG$-module $M$, we let $\cx_\cG(M)$ denote the {\it
complexity} of $M$. By definition, $\cx_\cG(M) = \gr(P^\bullet)$ coincides with the polynomial rate of growth of a minimal projective resolution $P^\bullet$ of $M$. Recall that the {\it
growth} of a sequence $\cV := (V_n)_{n\ge 0}$ of finite-dimensional $k$-vector spaces is defined via
\[ \gr(\cV) := \min \{c \in \NN_0\cup \{\infty\} \ ; \ \exists \, \lambda > 0 \ \text{such that} \ \dim_kV_n \le \lambda n^{c-1} \ \ \forall \ n \ge 0\}.\]
If $M$ is a $\cG$-module, then
\[ \cx_\cG(M) = \gr((\Omega_\cG^n(M))_{n\ge 0}),\]
where $\Omega_\cG$ denotes the Heller operator of the self-injective algebra $k\cG$. In particular, a $\cG$-module $M$ is projective if and only if $\cx_\cG(M)=0$.

Recall that $\Omega_\cG$ induces an auto-equivalence on the stable category $\underline{\modd} \cG$, whose objects are those of $\modd \cG$ and whose morphisms spaces $\underline{\Hom}_\cG(M,N) = \Hom_\cG(M,N)/P(M,N)$ are the factor groups of $\Hom_\cG(M,N)$ by the subspace $P(M,N)$ of those morphisms that factor through a projective module.

Thanks to the Friedlander-Suslin Theorem \cite{FS}, we can associate to every finite-dimensional $\cG$-module $M$ its {\it cohomological support variety} $\cV_\cG(M)$. By definition,
$\cV_\cG(M)$ is the variety $Z(\ker \Phi_M)$, associated to the kernel of the canonical homomorphism
\[ \Phi_M : \HH^\bullet(\cG,k) \lra \Ext^ \ast_\cG(M,M) \ \ ; \ \ [f] \mapsto [f\!\otimes\! \id_M].\]
It is well-known that $\cV_\cG(M)$ is a conical variety such that
\[ \dim \cV_\cG(M) = \cx_\cG(M).\]
The reader is referred to \cite{Be2} for basic properties of support varieties. We shall use \cite{Ja3} and \cite{ARS,ASS} as general references for representations of algebraic groups and associative algebras, respectively.

\bigskip

\section{An Upper Bound for the Complexity} \label{S:UB}
Throughout this section, we let $\cG$ denote an infinitesimal group, defined over an algebraically closed field $k$ of characteristic $p>0$. For such a group, the associated Hopf algebra
$k\cG$ coincides with the algebra $\Dist(\cG)$ of distributions on $\cG$.

In the sequel, all $\cG$-modules are assumed to be finite-dimensional. Given $r \ge 0$, we let $\cG_r$ be the r-th Frobenius kernel of $\cG$ and define the {\it height} of $\cG$ via
$\height(\cG) := \min \{r \in \NN_0 \ ; \ \cG_r = \cG\}$. The following result establishes an upper bound for $\cx_\cG(M)$ in terms of self-extensions.

\bigskip

\begin{Theorem} \label{UB1} Suppose that $\cG$ has height $r$. If $M$ is a $\cG$-module, then
\[ \cx_\cG(M) \le \dim_k \Ext^{2np^{r-1}}_\cG(M,M)\]
for every $n \ge 1$. \end{Theorem}

\begin{proof} We fix a natural number $n \in \NN$. According to \cite[(1.5)]{FS}, there exists a commutative, graded subalgebra $S \subseteq \Ext_\cG^\ast(M,M)$ of the Yoneda algebra
$\Ext^\ast_\cG(M,M)$ of $M$ such that

(a) \ $S$ is generated by $\bigoplus_{i=0}^{r-1} S_{2p^i}$, and

(b) \ $\Ext^\ast_\cG(M,M)$ is a finitely generated $S$-module.

\noindent
Thus, $S$ is finitely generated, and an integral extension of the subalgebra $T_{(n)} := k[S_{2np^{r-1}}]$, generated by the subspace $S_{2np^{r-1}}$ of homogeneous elements of
degree $2np^{r-1}$, cf.\ \cite[(9.1)]{Ma}. Owing to \cite[(4.5)]{Ei}, $S$ is a finitely generated $T_{(n)}$-module, so that $\Ext^\ast_\cG(M,M)$ also enjoys this property. In view
of \cite[(5.3.5)]{Be2}, passage to growths now yields
\[ \cx_{\cG}(M) = \gr(\Ext^\ast_\cG(M,M)) = \gr(T_{(n)}) \le \dim_k S_{2np^{r-1}} \le \dim_k \Ext_\cG^{2np^{r-1}}(M,M),\]
as desired. \end{proof}

\bigskip

\begin{Examples} Suppose that $p \ge 3$.

(1) For $r>0$, we consider the infinitesimal group $\cG := \SL(2)_1T_r$, whose factors are the first and r-th Frobenius kernels of $\SL(2)$ and its standard maximal torus $T \subseteq \SL(2)$ of diagonal matrices, respectively. We denote by $\alpha$ the positive root of $\SL(2)$ (relative to the Borel subgroup of upper triangular matrices) and recall that $\ZZ \lra X(T) \ ; \ n\mapsto \frac{n}{2}\alpha$ is an isomorphism between $\ZZ$ and the character group $X(T)$ of $T$. The character group of $T_r$ may thus be identified with $\ZZ/(p^r)$. Let $\lambda \in X(T_r)\setminus \{ip-1 \ ; \ 1 \le i \le p^{r-1}\}$ be a weight. Then the baby Verma module
\[ \widehat{Z}_1(\lambda) := \Dist(\SL(2)_1)\!\otimes_{\Dist(B_1)}\!k_\lambda\]
is a $\cG$-module of complexity $\cx_\cG(\widehat{Z}_1(\lambda)) = 1$. According to \cite[(4.5)]{Fa5} and its proof, we have
\[ \Omega^2_\cG(\widehat{Z}_1(\lambda)) \cong \widehat{Z}_1(\lambda)\! \otimes_k \! k_{p\bar{\alpha}} \cong \widehat{Z}_1(\lambda+p\bar{\alpha}),\]
where $\bar{\alpha} \in X(T_r)\cong X(T)/p^rX(T)$ denotes the restriction of the positive root $\alpha \in X(T)$ to $T_r$. Thus, $\bar{\alpha}$ corresponds to $2 \in \ZZ/(p^r)$. Since
each Verma module $\widehat{Z}_1(\lambda)$ has length $2$ with composition factors $\widehat{L}_1(\lambda)$ and $\widehat{L}_1(2p-2-\lambda)$, the choice of $\lambda$ yields
$\Hom_{\cG}(\widehat{Z}(\lambda+np\bar{\alpha}),\widehat{Z}(\lambda)) = (0)$ for $1 \le n \le p^{r-1}-1$. Consequently, we have
\[ \Ext^{2n}_{\cG}(\widehat{Z}_1(\lambda),\widehat{Z}_1(\lambda)) \cong \underline{\Hom}_\cG(\Omega^{2n}_{\cG}(\widehat{Z}_1(\lambda)),\widehat{Z}_1(\lambda)) \cong
\underline{\Hom}_\cG(\widehat{Z}_1(\lambda+np\bar{\alpha})),\widehat{Z}_1(\lambda)) = (0)\]
for each of these $n$, so that none of the Ext-groups $\Ext^{2n}_\cG(\widehat{Z}_1(\lambda),\widehat{Z}_1(\lambda))$ of degree $<2p^{r-1}$ provides an upper bound for the
complexity of the $\cG$-module $\widehat{Z}_1(\lambda)$.

(2) Let $\cG = \GG_{a(2)}$ be the second Frobenius kernel of the additive group $\GG_{a}$. It is well-known (cf. \cite[(3.5)]{Ev}) that
\[ \Ext^\ast_{\GG_{a(2)}}(k,k) \cong \HH^\ast(\GG_{a(2)},k) \cong k[X_1,X_2]\!\otimes_k\!\Lambda(Y_1,Y_2),\]
where the generators $X_1,X_2$ of the polynomial ring and $Y_1,Y_2$ of the exterior algebra have degrees $2$ and $1$, respectively. Consequently,
\[ \dim_k\Ext^{2p}_{\GG_{a(2)}}(k,k) = \dim_kk[X_1,X_2]_{2p}+ 2\dim_kk[X_1,X_2]_{2p-1} + \dim_kk[X_1,X_2]_{2p-2} = 2p+1,\]
while $\cx_{\GG_{a(2)}}(k)=2$. \end{Examples}

\bigskip
\noindent
Recall that a group scheme $\cG$ is referred to as {\it representation-finite} if and only if $\modd \cG$ has only finitely many isoclasses of indecomposable objects. An indecomposable
$\cG$-module is said to be {\it periodic} if there exists $n \ge 1$ such that $\Omega^n_\cG(M) \cong M$. The first part of the following result refines \cite[(7.6.1)]{SFB2}.

\bigskip

\begin{Corollary} \label{UB2} Let $\cG$ be an infinitesimal group of height $r$. Then the following statements hold:

{\rm (1)} \ A $\cG$-module $M$ is projective if and only if $\Ext_\cG^{2np^{r-1}}(M,M) = (0)$ for some $n \ge 1$.

{\rm (2)} \ The group $\cG$ is diagonalizable if and only if $\HH^{2np^{r-1}}(\cG,k) = (0)$ for some $n \ge 1$.

{\rm (3)} \ The group $\cG$ is representation-finite if and only if $\dim_k \HH^{2np^{r-1}}(\cG,k) \le 1$ for some $n \ge 1$. \end{Corollary}

\begin{proof} (1) According to (\ref{UB1}) we have $\cx_\cG(M) = 0$, so that $M$ is projective.

(2) By part (1), the trivial $\cG$-module is projective, so that the algebra of measures $k\cG$ of $\cG$ is semi-simple. Our assertion now follows from Nagata's theorem
\cite[(IV,\S3,(3.6))]{DG}.

(3) If $\cG$ is representation-finite and not diagonalizable, then the trivial $\cG$-module $k$ is periodic. By the Friedlander-Suslin Theorem \cite[(1.5)]{FS}, the even cohomology ring
$\HH^\bullet(\cG,k)$ is generated in degrees $2p^i$, with $i \in \{0,\ldots,r\!-\!1\}$. In view of \cite[(5.10.6)]{Be2}, the period of $k$ divides $2p^{r-1}$, so that
$\Omega_\cG^{2p^{r-1}}(k)\cong k$. This readily yields $\dim_k \HH^{2p^{r-1}}(\cG,k) = 1$.

Let $n \in \NN$ be such that $\dim_k \HH^{2np^{r-1}}(\cG,k) \le 1$. Then Theorem \ref{UB1} implies $\cx_\cG(k) \le 1$ and our assertion is a consequence of \cite[(1.1),(2.7)]{FV1}. \end{proof}

\bigskip

\begin{Corollary} \label{UB3} Let $M$ be a $\cG$-module. Then the following statements hold:

{\rm (1)} \ If $M$ is simple and such that $\Top(\Omega^{2np^{r-1}}_\cG(M))$ is simple for some $n \ge 1$, then $\Omega^{2np^{r-1}}_\cG(M) \cong M$.

{\rm (2)} \  If $M$ is indecomposable of length $\ell(M)=2$ and $\ell(\Omega^{2np^{r-1}}_\cG(M)) \le 2$ for some $n \ge 1$, then $M$ is projective or periodic. \end{Corollary}

\begin{proof} By general theory, we have
\[ \Ext^{2np^{r-1}}_\cG(M,M) \cong \underline{\Hom}_\cG(\Omega^{2np^{r-1}}_\cG(M),M).\]
(1) By assumption, the modules $M$ and $S := \Top(\Omega^{2np^{r-1}}_\cG(M))$ are simple, so that Schur's Lemma implies
\[ \dim_k \Ext_\cG^{2np^{r-1}}(M,M) \cong \dim_k \Hom_\cG(S,M) = \delta_{[S],[M]},\]
where the brackets indicate isomorphism classes. If $S\not \cong M$, then Corollary \ref{UB2} yields that $M$ is projective, whence $S=(0)$, a contradiction. Consequently, $S \cong M$, as
desired.

(2) If $M$ is indecomposable of length $2$, then either $\Omega^{2np^{r-1}}_\cG(M)
\cong M$ and $M$ is periodic, or $\Hom_\cG(\Omega^{2np^{r-1}}_\cG(M),M) \cong \Hom_\cG(\Omega^{2np^{r-1}}_\cG(M),\Soc(M)) \cong
\Hom_\cG(\Top(\Omega^{2np^{r-1}}_\cG(M)),\Soc(M))$. Schur's Lemma in conjunction with (\ref{UB1}) then implies $\cx_\cG(M) \le 1$, so that $M$ is periodic or projective (cf.\
\cite[(5.10.4)]{Be2}). \end{proof}

\bigskip

\section{Varieties for $\GG_{a(r)}$-Modules}\label{S:VG}
In Section \ref{S:PH} we shall study questions concerning the periodicity of $\cG$-modules by considering their rank varieties of infinitesimal one-paramenter subgroups of $\cG$. These are
defined via the groups
\[ \GG_{a(r)} := {\rm Spec}_k(k[T]/(T^{p^r}))\ \ \ \ \ \ \ (r \ge 1).\]
We denote the canonical generator of the coordinate ring by $t$. The algebra of measures $k\GG_{a(r)}$ of $\GG_{a(r)}$ is isomorphic to
\[k[U_0,\ldots,U_{r-1}]/(U_0^p,\ldots,U_{r-1}^p),\]
with $U_i + (U_0^p,\ldots,U_{r-1}^p)$ corresponding to the linear form $u_i$ on $k[\GG_{a(r)}]$ that sends $t^j$ onto $\delta_{p^i,j}$.

{\it Throughout, we assume that $p\ge 3$}. We write $\HH^\ast(k[u_i],k) = k[x_{i+1}]\!\otimes_k\!\Lambda(y_i)$ with $\deg(x_{i+1}) = 2$ and $\deg(y_i) = 1$. The K\"unneth
formula then provides an isomorphism
\[ k[x_1,\ldots,x_r]\!\otimes_k\!\Lambda(y_0,\ldots,y_{r-1}) \cong \HH^\ast(\GG_{a(r)},k) \cong \bigotimes_{i=0}^{r-1}\HH^\ast(k[u_i],k)\]
of graded $k$-algebras, where $\Lambda(y_0,\ldots,y_{r-1})$ denotes the exterior algebra in the variables $y_0,\ldots,y_{r-1}$. In this identification,
\[\HH^\ast(k[u_i],k) \cong k\!\otimes_k \cdots \otimes_k\! k\! \otimes_k\!\HH^\ast(k[u_i],k)\!\otimes_k \cdots \otimes_k\!k\]
corresponds to the image of the map $\HH^\ast(k[u_i],k) \lra \HH^\ast(\GG_{a(r)},k)$, defined by the algebra homomorphism $k\GG_{a(r)} \lra k[u_i] \ ; \ u_j \mapsto \delta_{i,j}u_i$.

\bigskip
\noindent
The above notation derives from the grading associated to the action of a torus $T$ on $\GG_{a(r)}$. If $T$ operates via a character $\alpha : T \lra k^\times$, i.e.,
\[ t\dact x = \alpha(t)x \ \ \ \ \ \ \ \ \forall \ t \in T, \ x \in \GG_{a(r)},\]
then, thanks to  \cite[(I.4.27)]{Ja3} (see also \cite[(4.1)]{CPSK}), the induced action of $T$ on $\HH^\ast(\GG_{a(r)},k)$ can be computed as follows:

\bigskip

\begin{Lemma} \label{OPG1} The following statements hold:

{\rm (1)} \ $x_i \in \HH^\bullet(\GG_{a(r)},k)_{-p^i\alpha}$ for $1 \le i \le r$.

{\rm (2)} \ $y_i \in \HH^\ast(\GG_{a(r)},k)_{-p^i\alpha}$ for  $0 \le i \le r\!-\!1$. \hfill $\square$ \end{Lemma}

\bigskip
\noindent
Given $s \le r$, we consider the standard embedding $\GG_{a(s)} \hookrightarrow \GG_{a(r)}$, whose comorphism is the projection
\[ \pi : k[T]/(T^{p^r}) \lra  k[T]/(T^{p^s}) \ \ ; \ \ f + (T^{p^r}) \mapsto f + (T^{p^s}).\]
The resulting embedding of algebras of measures is thus given by
\[ k\GG_{a(s)} \lra k\GG_{a(r)} \ \ ; \ \ u_i \mapsto u_i \ \ \ \ \ \ 0 \le i \le s-1.\]
Let $F : \GG_{a(r)} \lra \GG_{a(r-1)} \  ; \ x \mapsto x^p$ be the Frobenius homomorphism. Setting $u_{-1} := 0$, we see that the corresponding homomorphism of Hopf algebras is given
by
\[ F : k\GG_{a(r)} \lra k\GG_{a(r-1)} \ \ ; \ \ u_i \mapsto u_{i-1} \ \ \ \ \ \ 0 \le i \le r-1.\]
We recall the notion of a $p$-point, introduced by Friedlander-Pevtsova \cite{FPe}. Let $\fA_p = k[X]/(X^p)$ be the truncated polynomial ring with canonical generator $u := X+(X^p)$.
An algebra homomorphism $\alpha : \fA_p \lra k\cG$ is a {\it $p$-point} of $\cG$ if

(a) \ $\alpha$ is left flat, and

(b) \ there exists an abelian unipotent subgroup $\cU \subseteq \cG$ such that $\im \alpha \subseteq k\cU$.

\noindent
If $\alpha : \fA_p \lra k\cG$ is an algebra homomorphism, then $\alpha^\ast : \modd \cG \lra \modd \fA_p$ denotes the associated pull-back functor. Two $p$-points $\alpha$ and $\beta$
are {\it equivalent} if for every $M \in \modd \cG$ the module $\alpha^\ast(M)$ is projective precisely when $\beta^\ast(M)$ is projective. We denote by $P(\cG)$ the space of equivalence
classes of $p$-points.

The cohomological interpretation of $p$-points is based on the induced algebra homomorphisms $\alpha^\bullet : \HH^\bullet(\cG,k) \lra \HH^\bullet(\fA_p,k)$.  Following Friedlander-Pevtsova, we define for $M \in \modd \cG$ the {\it $p$-support} of $M$ via
\[P(\cG)_M := \{ [\alpha] \in P(\cG) \ ; \ \alpha^\ast(M) \ \text{is not projective}\}.\]
According to \cite[(3.10),(4.11)]{FPe}, the sets $P(\cG)_M$ are the closed sets of a noetherian topology on $P(\cG)$ and the map
\[ \Psi_\cG : P(\cG) \lra \Proj(\cV_\cG(k)) \ \ ; \ \ [\alpha] \mapsto \ker \alpha^\bullet\]
is a homeomorphism with $P(\cG)_M = \Psi_\cG(\Proj(\cV_\cG(M))$ for every $M \in \modd \cG$.  Moreover, $\Psi_\cG$ is natural with respect to flat maps $\cH \lra \cG$ of
finite group schemes.

Given $f \in \HH^\bullet(\GG_{a(r)},k)$, we let $Z(f)$ be the zero locus of $f$, that is, the set of the maximal ideals of $\HH^\bullet(\GG_{a(r)},k)$ containing $f$.

\bigskip

\begin{Lemma} \label{OPG2} Let $N$ be a $\GG_{a(r)}$-module such that $\cV_{\GG_{a(r)}}(N) \ne \{0\} = \cV_{\GG_{a(r-1)}}(N)$.
Then we have
\[ Z(x_r)\cap \cV_{\GG_{a(r)}}(N) \subsetneq \cV_{\GG_{a(r)}}(N).\] \end{Lemma}

\begin{proof} Let $\alpha : \fA_p \lra k\GG_{a(r)}$ be a $p$-point such that $[\alpha] \in P(\GG_{a(r)})_N$. In view of \cite[(2.2)]{FPe}, we may assume that $\alpha$ sends the generator
$u \in \fA_p$ onto
\[ \alpha(u) = a_0u_0+\cdots+ a_{r-1}u_{r-1} \ \ \  \ (a_i \in k).\]
Since $P(\GG_{a(r-1)})_N = \emptyset$ (cf.\ \cite[(4.11)]{FPe}), we conclude that $a_{r-1} \ne 0$. As noted in \cite[(1.13(2))]{SFB1}, the iterated Frobenius homomorphism $F^{r-1} :
\GG_{a(r)} \lra \GG_{a(1)}$ induces a map
\[ (F^{r-1})^\bullet : \HH^\bullet(\GG_{a(1)},k) \lra \HH^\bullet(\GG_{a(r)},k) \ \ ; \ \  x_1 \mapsto x_r.\]
Since $F^{r-1}(\alpha(u)) = a_{r-1}u_0$, the map $F^{r-1}\circ \alpha$ is an isomorphism of $k$-algebras. Consequently,
$\alpha^\bullet \circ (F^{r-1})^\bullet$ is bijective, and
\[ \alpha^\bullet(x_r) = \alpha^\bullet((F^{r-1})^\bullet(x_1)) \ne 0.\]
Since $\ker \alpha^\bullet \in {\rm Proj}(\cV_{\GG_{a(r)}}(N))$, it follows that the radical ideal $I_N \subseteq \HH^\bullet(\GG_{a(r)},k)$ defining $\cV_{\GG_{a(r)}}(N)$ is contained
in $\ker \alpha^\bullet$. By the above, the image of $x_r$ in the coordinate ring $k[\cV_{\GG_{a(r)}}(N)]$ is not zero, and Hilbert's Nullstellensatz provides a maximal ideal $\fM \unlhd
\HH^\bullet(\GG_{a(r)},k)$ such that $\fM \supseteq I_N$ and $x_r \not \in \fM$. Consequently,
\[\fM \in \cV_{\GG_{a(r)}}(N) \setminus (Z(x_r)\cap \cV_{\GG_{a(r)}}(N)),\]
as desired.  \end{proof}

\bigskip
\noindent
Let $\cG$ be an algebraic $k$-group. In \cite[\S1]{SFB1} the authors introduce the affine algebraic scheme $V_r(\cG)$ of infinitesimal one-parameter subgroups. By definition,
\[ V_r(\cG) = \mathcal{HOM}(\GG_{a(r)},\cG)\]
is the homomorphism scheme, cf.\ \cite[p.18]{Wa}. Owing to \cite[(1.14)]{SFB1}, there exists a homomorphism
\[ \Psi^r_\cG : \HH^\bullet(\cG,k) \lra k[V_r(\cG)]\]
of commutative $k$-algebras which multiplies degrees by $\frac{p^r}{2}$. Moreover, the map $\Psi^r_\cG$ is natural in $\cG$.

Let $s\le r$. We conclude this section with a basic observation concerning the map
\[ \Psi^r_{\GG_{a(s)}} : \HH^\bullet(\GG_{a(s)},k) \lra k[V_r(\GG_{a(s)})].\]
In view of \cite[(1.10)]{SFB1} (and its proof), the coordinate ring
\[k[V_r(\GG_{a(s)})] \cong k[T_{r-s},\ldots,T_{r-1}]\]
is reduced with $\ZZ$-grading given by $\deg(T_i) = p^i$ (see also \cite[(1.12)]{SFB1}).

\bigskip

\begin{Lemma} \label{OPG3} We have $\Psi^r_{\GG_{a(s)}}(x_i) = T_{r-i}^{p^i}$ for $1 \le i \le s$. \end{Lemma}

\begin{proof} This is a direct consequence of the proof of \cite[(6.5)]{SFB2}. \end{proof}

\bigskip

\section{Projective Height and Periodicity}\label{S:PH}
Let $k$ be an algebraically closed field of characteristic $\Char(k)=p\ge 3$. Throughout this section, we let $\cG$ be an infinitesimal $k$-group of height $\height(\cG) = r$. To each
non-projective $\cG$-module $M \in \modd \cG$ we associate its projective height $\ph(M)$. This numerical invariant, which will later be seen to be constant on the components of the stable
Auslander-Reiten quiver of $\cG$, provides information on the period of periodic modules.

\bigskip

\begin{Definition} A subgroup $\cU \subseteq \cG$ is called {\it elementary abelian} if there exists $s \in \NN$ such that $\cU \cong \GG_{a(s)}$. We let $\fE(\cG)$ be the set of
elementary abelian subgroups of $\cG$. \end{Definition}

\bigskip

\begin{Definition} Let $M$ be a $\cG$-module, $\cH \subseteq \cG$ be a closed subgroup. Then
\[ \ph_\cH(M) := \left\{ \begin{array}{cl} \min \{ 1 \le t \le r \ ; \ M|_{\cH_t} \ \text{is not projective}\}  & \text{if} \ M|_\cH \ \text{is not projective,}\\ 0 & \text{otherwise}
\end{array} \right.\]
is called the {\it projective height of $M$ relative to $\cH$}. \end{Definition}

\bigskip
\noindent
Let $M$ be a non-projective $\cG$-module. According to \cite[(7.6)]{SFB2}, there exists an elementary abelian subgroup $\cU \in \fE(\cG)$ such that $M|_\cU$ is not projective. This
motivates the following definition:

\bigskip

\begin{Definition}  Let $M$ be a non-projective $\cG$-module. Then
\[ \ph(M) := \max_{\cU \in \fE(\cG)} \ph_\cU(M)\]
is referred to as the {\it projective height} of $M$. \end{Definition}

\bigskip
\noindent
Let $M$ be a $\cG$-module of projective height $\ph(M)=t>0$. Then there exists a subgroup $\cU \subseteq \cG$ with $\cU \cong \GG_{a(t-1)}$ and $M|_\cU$ being projective.
Consequently, $M$ is a free module of the $p^{t-1}$-dimensional algebra $k\cU$, so that $p^{t-1}\!\mid\!\dim_kM$.

\bigskip

\begin{Example} Let $G$ be a reductive group. We consider the Steinberg module $\St_d$, which view as a $G_r$-module. For dimension reasons, $\St_d|_{G_s}$ is not projective for $d<s\le r$, while \cite[(II.11.8)]{Ja3} shows that $\St_d|_{G_d}$ is projective.  Let $d<r$. In view of \cite[(7.6)]{SFB2} we thus have $\ph(\St_d) = d\!+\!1$. \end{Example}

\bigskip
\noindent
Given a commutative $k$-algebra $A$, we denote by $A_{\rm red}$ the associated reduced algebra. A homomorphism $f : A \lra B$ of commutative $k$-algebras induces a homomorphism
$f_{\rm red} : A_{\rm red} \lra B_{\rm red}$ of reduced $k$-algebras. The commutative algebras relevant for our purposes are the even cohomology rings $\HH^\bullet(\cG,k)$ and
$\HH^\bullet(\cU,k)$, where $\cU \cong \GG_{a(s)}$ is an elementary abelian subgroup of $\cG$. We let
\[ {\rm res} : \HH^\bullet(\cG,k) \lra \HH^\bullet(\cU,k)\]
be the canonical restriction map, and recall that the canonical inclusion $k[x_1,\ldots,x_s] \lra \HH^\bullet(\cU,k)$ induces an isomorphism $k[x_1,\ldots,x_s] \cong \HH^\bullet(\cU,k)
_{\rm red}$, see \cite{CPSK}.

Let $\cH \subseteq \cG$ be a closed subgroup. By virtue of \cite[(5.4.1)]{SFB2}, the canonical restriction map ${\rm res} : \HH^\bullet(\cG,k) \lra \HH^\bullet(\cH,k)$ induces a morphism ${\rm res}^\ast : \cV_{\cH}(k) \lra \cV_\cG(k)$ which maps $\cV_\cH(k)$ homeomorphically onto its image. Bearing this in mind, we shall often identify $\cV_\cH(k)$ topologically with a closed subvariety of $\cV_\cG(k)$.

\bigskip

\begin{Proposition} \label{PH1} Let $M$ be a $\cG$-module, $\cU \cong \GG_{a(s)}$ be an elementary abelian subgroup of $\cG$.

{\rm (1)} \ If $M|_\cU$ is not projective, then there exists $\zeta \in \HH^{2p^{r-\ph_\cU(M)}}(\cG,k)_{\rm red}$ such that
\[ Z(\zeta) \cap \cV_\cG(M) \subsetneq \cV_\cG(M).\]

{\rm (2)} \ If $\cV_{\cU}(k) \subseteq \cV_\cG(M)$, then there exists $\zeta \in \HH^{2p^{r-s}}(\cG,k)_{\rm red}$ such that
\[ Z(\zeta) \cap \cV_\cG(M) \subsetneq \cV_\cG(M).\] \end{Proposition}

\begin{proof} Let $1 \le t \le s$. Owing to \cite[(1.14)]{SFB1} we have a commutative diagram
\[ \begin{CD} \HH^\bullet(\cG,k) @> \Psi^r_{\cG} >> k[V_r(\cG)]\\
@V{\rm res} VV @V\pi VV\\
 \HH^\bullet(\cU_t,k) @>\Psi^r_{\cU_t}>> k[V_r(\cU_t)], \end{CD} \]
of homomorphisms of graded, commutative $k$-algebras, where the horizontal arrows multiply degrees by $\frac{p^r}{2}$. Thanks to \cite[(1.5)]{SFB1}, the map $\pi$ is surjective, and
\cite[(1.12)]{SFB1} shows that it respects degrees.

Since $\cU_t \cong \GG_{a(t)}$ we may consider the element $T_{r-t} \in k[V_r(\cU_t)]_{p^{r-t}}$. As $\pi$ is surjective, we can find $v_{r-t} \in k[V_r(\cG)]_{p^{r-t}}$ with
$\pi(v_{r-t}) = T_{r-t}$. According to \cite[(5.2)]{SFB2}, we have $v_{r-t}^{p^r} \in {\rm im}\, \Psi^r_{\cG}$, so that there exists $w \in  \HH^{2p^{r-t}}(\cG,k)$ with
$\Psi^r_{\cG}(w) = v_{r-t}^{p^r}$. In light of (\ref{OPG3}), we thus obtain
\[ \Psi^r_{\cU_t}({\rm res}(w)) = \pi(\Psi^r_{\cG}(w)) = \pi(v_{r-t}^{p^r}) = T_{r-t}^{p^r} = \Psi^r_{\cU_t}(x_t^{p^{r-t}}).\]
Thanks to \cite[(5.2)]{SFB2}, we conclude that the residue class $\zeta_t := \bar{w} \in \HH^{2p^{r-t}}(\cG,k)_{\rm red}$ satisfies
\[ (\ast) \ \ \ \ \ \ \ \ \ {\rm res}_{\rm red}(\zeta_t) = \bar{x}_t^{p^{r-t}}.\]
If we identify $\cV_{\cU_t}(k)$ with its image under the morphism ${\rm res}^\ast :  \cV_{\cU_t}(k) \lra \cV_\cG(k)$, whose comorphism is the restriction map ${\rm res}_{\rm red} :
\HH^\bullet(\cG,k)_{\rm red} \lra \HH^\bullet(\cU_t,k)_{\rm red}$, then ($\ast$) implies
\[ Z(\zeta_t) \cap \cV_{\cU_t}(k) = Z(x_t).\]
Let $t := \ph_\cU(M)$. In view of \cite[(7.1)]{SFB2}, the assumption $Z(\zeta_t)\cap \cV_\cG(M) = \cV_\cG(M)$ yields
\[ \cV_{\cU_t}(M) = \cV_\cG(M)\cap \cV_{\cU_t}(k) = Z(\zeta_t)\cap \cV_{\cU_t}(k) \cap \cV_{\cU_t}(M) = Z(x_t) \cap \cV_{\cU_t}(M),\]
which, by choice of $t$, contradicts (\ref{OPG2}). This concludes the proof of (1).

For the proof of (2), we set $t:= s$, so that $\cU_t = \cU \cong \GG_{a(s)}$.  Consider the homomorphism
\[ \Phi_M : \HH^\bullet(\cG,k) \lra \Ext^\ast_{\cG}(M,M)\ \ ; \ \ [f] \mapsto [f\!\otimes\! \id_{M}].\]
We let $A := \HH^\bullet(\cG,k)/\sqrt{\ker \Phi_M}$ be the coordinate ring of the support variety $\cV_{\cG}(M)$ and denote by
\[ \iota^\ast : \HH^\bullet(\cG,k)_{\rm red} \lra A\]
the canonical projection map. According to our convention, the inclusion $\cV_\cU(k) \subseteq \cV_\cG(M)$ means that the image of the morphism ${\rm res}^\ast : \cV_{\cU}(k) \lra
\cV_\cG(k)$ is contained in $\cV_{\cG}(M)$. Thus, the map ${\rm res}^\ast$ factors through the inclusion $\iota : \cV_{\cG}(M) \hookrightarrow \cV_{\cG}(k)$, so that there exists a
homomorphism $\gamma^\ast : A \lra \HH^\bullet(\cU,k)_{\rm red}$ with
\[ {\rm res}_{\rm red} = \gamma^\ast \circ \iota^\ast.\]
By ($\ast$), we can find $\zeta := \zeta_s \in \HH^{2p^{r-s}}(\cG,k)_{\rm red}$ such that ${\rm res}_{\rm red}(\zeta) \ne 0$. Consequently, $\iota^\ast(\zeta) \ne 0$, and Hilbert's
Nullstellensatz provides a maximal ideal $\fM \supseteq \sqrt{\ker\Phi_M}$ which does not contain $\zeta$. This implies
\[ Z(\zeta)\cap \cV_{\cG}(M) \subsetneq \cV_{\cG}(M),\]
as desired. \end{proof}

\bigskip
\noindent
We record a consequence concerning modules of complexity $1$, which generalizes \cite[(2.5)]{Fa1}. The proof employs a method of Carlson (cf.\ \cite{Ca3}), which is based on the following
construction: By general theory, a cohomology class $\zeta \in \HH^{2n}(\cG,k)\setminus \{0\}$ corresponds to an element $\hat{\zeta} \in \Hom_\cG(\Omega_\cG^{2n}(k),k)\setminus
\{0\}$. We let
\[ L_\zeta := \ker \hat{\zeta}\]
be the {\it Carlson module} of $\zeta$.

\bigskip

\begin{Corollary}\label{PH2} Let $M$ be an indecomposable $\cG$-module of complexity $\cx_\cG(M) = 1$. Then we have
\[ \Omega_\cG^{2p^{r-\ph(M)}}(M) \cong M.\] \end{Corollary}

\begin{proof} Let $t := \ph(M) = \ph_\cU(M)$ for some $\cU \in \fE(\cG)$. Owing to (\ref{PH1}(1)), we can find an element $\zeta \in \HH^{2p^{r-t}}(\cG,k)\setminus\{0\}$ such that
\[  Z(\zeta) \cap \cV_\cG(M) \subsetneq \cV_\cG(M).\]
Since $M$ is indecomposable, the variety $\cV_\cG(M)$ is a line (see \cite[(7.7)]{SFB2}). Thanks to \cite[p.755]{SFB2} we have $\cV_\cG(L_\zeta) = Z(\zeta)$, so that an application of
\cite[(7.2)]{SFB2} gives
\[ \{0\} = Z(\zeta) \cap \cV_\cG(M) = \cV_\cG(L_\zeta\!\otimes_k\!M).\]
Consequently, the module $L_\zeta \!\otimes_k\! M$ is projective and the exact sequence
\[ (0) \lra L_\zeta\!\otimes_k \!M \lra \Omega_\cG^{2p^{r-t}}(M) \oplus (\text{proj.}) \lra M \lra (0),\]
obtained by tensoring the sequence defined by $\hat{\zeta}$ with $M$, splits. By comparing projective-free summands of $\Omega_\cG^{2p^{r-t}}(M) \oplus (\text{proj.}) \cong
(L_\zeta\!\otimes_k\! M) \oplus M$, we arrive at $\Omega_\cG^{2p^{r-t}}(M) \cong M$. \end{proof}

\bigskip
\noindent
A $\cG$-module $M$ is called {\it periodic}, provided there exists $n \in \NN$ such that $\Omega_\cG^n(M) \cong M$. In that case,
\[ {\rm per}(M):= \min \{n \in \NN \ ; \ \Omega^{n}_\cG(M) \cong M\}\]
is called the {\it period} of $M$. By Corollary \ref{PH2}, the number ${\rm per}(M)$ divides $2p^{r-\ph(M)}$. The following examples show that the latter number may only provide a
rough estimate for ${\rm per}(M)$.

\bigskip

\begin{Examples} (1) Consider the infinitesimal group $\cG = \SL(2)_1T_r$, where $r \ge 2$. Then $\height(\cG) = r$ while any unipotent subgroup $\cU$ of $\cG$ has height $\le 1$. Consequently, every non-projective $\cG$-module has projective height $1$, so that (\ref{PH2}) yields $2p^{r-1}$ as an estimate for the period ${\rm per}(M)$ of a periodic module $M$.
According to \cite[(5.6)]{FV2} and \cite[(4.5)]{Fa5} ``most" periodic modules (namely those, whose rank varieties are not $T$-stable) have period $2$, while those with a $T$-stable
rank variety satisfy ${\rm per}(M) = 2p^{r-1}$.

(2) Let $\cU$ be a unipotent infinitesimal group of complexity $\cx_\cU(k)=1$. Owing to the main theorem of \cite{FRV} such groups can have arbitrarily large height. The corresponding
algebras $k\cU$ are truncated polynomial rings $k[X]/(X^{p^n})$, so that every indecomposable $\cU$-module $M$ is periodic with ${\rm per}(M) = 2$. If $\cU$ has height $r \ge 2$,
then only for those non-projective indecomposable modules $M$ with $\dim_kM = \ell \dim_k k\cU_{r-1}$ does Corollary \ref{PH2} provide the correct formula for ${\rm per}(M)$.
\end{Examples}

\bigskip

\begin{Remark} We shall see in Section \ref{S:AR} below that $2p^{r-\ph(M)}$ coincides with the period for graded modules of Frobenius kernels of reductive algebraic algebraic groups (see
Theorem \ref{MC3}). \end{Remark}

\bigskip
\noindent
Let $M$ be a $\cG$-module. By general theory, a homomorphism $\varphi : \GG_{a(r)} \lra \cG$ corresponds to a homomorphism $\varphi : k\GG_{a(r)} \lra k\cG$. If $M$ is a
$\cG$-module, then we denote by $M|_{k[u_{r-1}]}$ the pull-back of $M$ along the map $\varphi|_{k[u_{r-1}]}$. Following Suslin-Friedlander-Bendel \cite[\S6]{SFB2}, we let
\[ V_r(\cG)_M := \{ \varphi \in \Hom(\GG_{a(r)},\cG) \ ; \ M|_{k[u_{r-1}]} \ \text{is not projective}\}\]
be the {\it rank variety} of $M$.

\bigskip

\begin{Theorem} \label{PH3} Let $M$ be an indecomposable $\cG$-module such that $\cx_{\cG}(M) = 1$. Then
\[ \cU_M := \bigcup_{\varphi \in V_r(\cG)_M} \im \varphi \]
is an elementary abelian subgroup of $\cG$ such that $\ph(M) = \height(\cU_M) = \ph_{\cU_M}(M)$. \end{Theorem}

\begin{proof} Since $\cx_{\cG}(M) = 1$, an application of \cite[(6.8)]{SFB2} yields
\[\dim V_r(\cG)_M = \dim \cV_{\cG}(M) =1.\]
According to \cite[(6.1)]{SFB2}, the canonical action of $k^\times$ on $\GG_{a(r)}$ endows the variety $V_r(\cG)_M$ with the structure of a conical variety:
\[ (\alpha\dact\varphi)(x) := \varphi(\alpha.x) \ \ \ \ \ \ \forall \ \alpha \in k^\times, \, x \in k\GG_{a(r)}.\]
In view of $\alpha.u_{r-1} = \alpha^{p^{r-1}}u_{r-1}$, the group $k^\times$ acts simply on $V_r(\cG)_M \setminus \{0\}$. Since $M$ is indecomposable, Carlson's Theorem (see
\cite[(7.7)]{SFB2}) ensures that the variety $(V_r(\cG)_M\setminus\{0\})/k^\times$ is connected. Hence there exists $\varphi \in V_r(\cG)_M$ such  that
\[ V_r(\cG)_M = \{\alpha\dact \varphi \ ; \ \alpha \in k^\times\}\cup \{0\}.\]
As a result, $\cU_M = \im \varphi$ is a subgroup of $\cG$. Being isomorphic to the factor group $\GG_{a(r)}/\ker \varphi$, the group $\cU_M$ is elementary abelian.

Let $s := \ph_{\cU_M}(M)$. We propose to show that $s = \ph(M) = \height(\cU_M)$. Let $\cU$ be an elementary abelian subgroup of $\cG$ such that $M|_\cU$ is not projective. Owing
to \cite[(6.6)]{SFB2},  $V_r(\cU)_{M|_\cU} $ is a one-dimensional subvariety of the one-dimensional irreducible variety $V_r(\cG)_M$, whence $V_r(\cU)_{M|_\cU} = V_r(\cG)_M$. It
readily follows that $\cU_M \subseteq \cU$.

By applying this observation to the Frobenius kernels of $\cU_M$, we conclude that $s = \height(\cU_M)$. Thus, if $\cU$ is elementary abelian and $r' \in \NN$ is a natural number such that
the restriction $M|_{\cU_{r'}}$ of $M$ to the elementary abelian subgroup $\cU_{r'}$ is not projective, then $\cU_M \subseteq \cU_{r'}$, whence $r' \ge s$. Consequently, $\ell :=
\ph_\cU(M) \ge s$ and $\cU_M \subseteq \cU_\ell$. If $\ell>s$, then $\height(\cU_M)=s$ implies that $\cU_M \subseteq \cU_{\ell-1}$, while $M|_{\cU_{\ell-1}}$ is projective. This,
however, contradicts $M|_{\cU_M}$ being non-projective, so that $\ell=s$. As a result, we have
\[\height(\cU_M)= s = \ph(M),\]
as desired. \end{proof}

\bigskip
\noindent
We now specialize to the case, where $\cG=G_r$ is a Frobenius kernel of a smooth group scheme $G$. Then $G$ acts on $G_r$ via the adjoint representation, and we obtain an action of $G$
on $V_r(G)$. If $M$ is a $G_r$-module and $g \in G$, then $M^{(g)}$ denotes the $G_r$-module with underlying $k$-space $M$ and action
\[ x\dact m := \Ad(g)^{-1}(x)m \ \ \ \ \forall \ x \in G_r, m \in M.\]
One readily verrifies that
\[ V_r(G)_{M^{(g)}} = g\dact V_r(G)_M \ \ \ \ \forall \ g \in G.\]
We refer to a $G_r$-module as {\it $G$-stable} if $M^{(g)} \cong M$ for every $g\in G$. Clearly, the restriction $M|_{G_r}$ of a $G$-module $M$ is a $G$-stable $G_r$-module. Moreover,
if $G$ is connected, then every simple $G_r$-module is $G$-stable.

\bigskip

\begin{Corollary} \label{PH3.5} Let $G$ be a smooth group scheme, $M$ be an indecomposable, $G$-stable $G_r$-module of complexity $\cx_{G_r}(M)=1$. Then $\cU_M \unlhd G_r$ is a
normal subgroup of $G_r$. \end{Corollary}

\begin{proof} This follows directly from Theorem \ref{PH3} and the definition of the $G$-action on $V_r(G)_M$. \end{proof}

\bigskip
\noindent
We fix a maximal torus $T\subseteq G$ and denote by $X(T)$ its character group. Since $T$ acts diagonally on the Lie algebra $\fg := \Lie(G)$, there exists a finite subset $\Psi \subseteq
X(T)\setminus\{0\}$ such that
\[ \fg = \fg_0 \oplus \bigoplus_{\alpha \in \Psi} \fg_\alpha,\]
with $\fg_\alpha := \{x \in \fg \ ; \ t.x = \alpha(t)x \ \ \ \ \forall \ t \in T\}$ for $\alpha \in \Psi\cup\{0\}$. The elements of $\Psi$ are the {\it roots} of $G$.

Suppose that $\cU \subseteq G_r$ is a $T$-invariant elementary abelian subgroup. The maximal torus $T$ acts canonically on $\HH^\bullet(G_r,k)_{\rm red}$ and
$\HH^\bullet(\cU,k)_{\rm red}$. In light of (\ref{OPG1}), the weights of this action on the latter space are of the form $-n\alpha$ for $n \ge 0$. We thus have the following refinement of
Proposition \ref{PH1}:

\bigskip

\begin{Proposition} \label{PH4} Let $G$ be a smooth group scheme, $\cU \subseteq G_r$ be a $T$-invariant elementary abelian subgroup, $M$ be a $G_r$-module. Then there exists $\alpha
\in \Psi \cup \{0\}$ with the following properties:

{\rm (1)} \ If $M|_\cU$ is not projective, then there exists $\zeta \in (\HH^{2p^{r-\ph_\cU(M)}}(G_r,k)_{\rm red})_{-p^r\alpha}$ such that
\[ Z(\zeta) \cap \cV_{G_r}(M) \subsetneq \cV_{G_r}(M).\]

{\rm (2)} \ If $\cV_\cU(k) \subseteq \cV_{G_r}(M)$, then there exists $\zeta \in (\HH^{2p^{r-\height(\cU)}}(G_r,k)_{\rm red})_{-p^r\alpha}$ such that
\[ Z(\zeta) \cap \cV_{G_r}(M) \subsetneq \cV_{G_r}(M).\]\end{Proposition}

\begin{proof} Since $\cU \cong \GG_{a(s)}$ is $T$-invariant, the diagonalizable group $T$ acts on $\cU$ via a character $\alpha \in X(T)$:
\[ t\dact u = \alpha(t)u \ \ \ \ \forall \ u \in \cU.\]
Hence $T$ also acts on $\Lie(\cU) \subseteq \fg$ via $\alpha$, so that $\alpha \in \Psi\cup \{0\}$.

We return to the proof of Proposition \ref{PH1}. For $t \le r$ we found elements $\zeta'_t \in \HH^{2p^{r-t}}(G_r,k)_{\rm red}$ satisfying
\[ {\rm res}_{\rm red}(\zeta'_t) = \bar{x}_t^{p^{r-t}}.\]
Owing to (\ref{OPG1}), the element $\bar{x}_t^{p^{r-t}}$ belongs to $(\HH^{2p^{r-t}}(\cU,k)_{\rm red})_{-p^r\alpha}$. Since the map ${\rm res}_{\rm red}$ is $T$-equivarant, we may replace $\zeta'_t$ by its homogeneous component $\zeta_t$ of degree $-p^r\alpha$, so that
\[ {\rm res}_{\rm red}(\zeta_t) = \bar{x}_t^{p^{r-t}} \ \ \text{for some} \ \ \zeta_t \in (\HH^{2p^{r-t}}(G_r,k)_{\rm red})_{-p^r\alpha}.\]
We may now adopt the arguments of the proof of (\ref{PH1}) verbatim to obtain our result. \end{proof}

\bigskip
\noindent
Now assume $G$ to be reductive, with Borel subgroup $B=UT$ and sets $\Psi^+$ and $\Sigma$ of positive roots and simple roots, respectively. As usual, $\rho$ denotes the half-sum of the positive roots. Given $d \ge 0$, we recall that
\[ \St_d := L((p^d\!-\!1)\rho)\]
denotes the $d$-th {\it Steinberg module}. By definition, $\St_d$ is a simple $G$-module with $\St_0 \cong k$, the trivial $G$-module. In view of \cite[(3.18(4))]{Ja3} and
\cite[(II.10.2)]{Ja3}
\[ \St_d \cong L_d((p^d\!-\!1)\rho) \cong Z_d((p^d\!-\!1)\rho)\]
is a simple, projective $G_d$-module.

Let $M$ be a $G_r$-module. If $\alpha \in \Psi$ is a root, we define
\[ \widehat{\ph}_\alpha(M) := \left\{ \begin{array}{cl} \ph_{(U_\alpha)_r}(M) & \text{if} \ M|_{(U_\alpha)_r} \ \text{is not projective,}\\ \infty & \text{otherwise,}\end{array} \right.\]
where $U_\alpha \subseteq G$ is the root subgroup of $\alpha$. Moreover, for a subset $\Phi \subseteq \Psi$, we put
\[ \widehat{\ph}_\Phi(M) := \min_{\alpha \in \Phi} \widehat{\ph}_\alpha(M).\]
Let $Y(T)$ be the set of co-characters of $T$ and denote by $\langle\, , \, \rangle : X(T)\times Y(T) \lra \ZZ$ the canonical pairing. For $s \in \NN_0$, we put
\[ \Psi^s_\lambda := \{\alpha\in \Psi \ ; \  \langle\lambda\!+\!\rho, \alpha^\vee \rangle \in p^s \ZZ \}.\]
Here $\alpha^\vee \in Y(T)$ denotes the root dual to $\alpha$. Owing to \cite[(2.7)]{Ja1}, the set $\Psi^s_\lambda$ is a subsystem of $\Psi$ whenever the prime number $p$ is good for
$G$.

Given $\lambda \in X(T)$ with $\langle \lambda+\rho,\alpha^\vee\rangle \ne 0$ for some $\alpha \in \Psi$, we define the \emph{depth} of $\lambda$ via
\[ \dep(\lambda) := \min\{s \in \NN_0 \ ; \  \Psi_\lambda^s \ne \Psi\},\]
and put $\dep(\lambda) = \infty$ otherwise, see \cite{FR}. The following result links the depth of a weight to the projective height $\widehat{\ph}_\Sigma(Z_r(\lambda))$ of the baby
Verma module
\[ Z_r(\lambda) := kG_r\!\otimes_{kB_r}\!k_\lambda.\]
Suppose that $G$ is a smooth group scheme that is defined over the Galois field $\FF_p$. Let $M$ be a $G$-module. Given $d \in \NN$, we let $M^{[d]}$ be the $G$-module with
underlying $k$-space $M$ and action defined via pull-back along the iterated Frobenius endomorphism $F^d : G \lra G$.

\bigskip

\begin{Proposition} \label{PH5} Suppose that $G$ is semi-simple, simply connected, defined over $\FF_p$ and that $p$ is good for $G$. Let $\lambda \in X(T)$ be a weight such that $\dep(\lambda) \le r$. Then the following statements hold:

{\rm (1)} \ We have $\cV_{(U_\alpha)_r}(\St_{\dep(\lambda)-1}) \subseteq \cV_{G_r}(Z_r(\lambda))$ for every $\alpha \in \Sigma\setminus \Psi^{\dep(\lambda)}_\lambda$.

{\rm (2)} \ We have $\dep(\lambda) = \widehat{\ph}_\Sigma(Z_r(\lambda))$. \end{Proposition}

\begin{proof} Let $\dep(\lambda) = d+1$ with $d\ge 0$. Owing to \cite[(6.2)]{FR}, there exists a weight $\mu$ of depth $1$ and an isomorphism
\[ Z_r(\lambda) \cong Z_{r-d}(\mu)^{[d]}\!\otimes_k\!\St_d\]
of $G_r$-modules. Moreover, we have $\lambda = p^d \mu +(p^d-1)\rho$. According to \cite[(5.2)]{FR}, the inclusion
\[ \cV_{(U_\alpha)_{r-d}}(k) \subseteq \cV_{G_{r-d}}(Z_{r-d}(\mu))\]
holds for every $\alpha \in \Sigma\setminus \Psi^1_\mu$.

Consider the iterate $F^d : G_r \lra G_{r-d}$ of the Frobenius endomorphism. There results a commutative diagram
\[ \begin{CD} \cV_{(U_\alpha)_{d+1}}(\St_d) @> F^d >> \cV_{(U_\alpha)_1}(k)\\
@V{\rm res}^\ast VV @V{\rm res}^\ast VV\\
 \cV_{(U_\alpha)_r}(\St_d) @> F^d >> \cV_{(U_\alpha)_{r-d}}(k)\\
  @V{\rm res}^\ast VV @V{\rm res}^\ast VV\\
\cV_{G_r}(\St_d) @> F^d >> \cV_{G_{r-d}}(k),\\
\end{CD} \]
where the vertical arrows are the canonical inclusions induced by the restriction maps. Thanks to \cite[(6.2(b))]{FR}, the lower horizontal map is an isomorphism sending $\cV_{G_r}(Z_r(\lambda))$ onto $\cV_{G_{r-d}}(Z_{r-d}(\mu))$.

(1) Let $\alpha \in \Sigma\setminus\Psi_\lambda^{d+1}$. Then we have
\[ p^d\langle\mu+\rho,\alpha^\vee \rangle = \langle \lambda +\rho, \alpha^\vee \rangle \not \in p^{d+1}\ZZ,\]
so that $\langle \mu+\rho,\alpha^\vee\rangle \not \in p\ZZ$, whence $\alpha \in \Sigma\setminus \Psi^1_\mu$. In view of the identification discussed at the beginning of this section,
the inclusion $\cV_{(U_\alpha)_{r-d}}(k) \subseteq \cV_{G_{r-d}}(Z_{r-d}(\mu))$ in conjunction with the above diagram now implies $\cV_{(U_\alpha)_r}(\St_d) \subseteq \cV_{G_r}(Z_r(\lambda))$.

(2) Given $\alpha \in \Sigma \setminus \Psi^{d+1}_\lambda$, part (1) yields $\cV_{(U_\alpha)_r}(\St_d) \subseteq \cV_{G_r}(Z_r(\lambda))$. The diagram above
then implies
\begin{eqnarray*}
\cV_{(U_\alpha)_{d+1}}(\St_d) & = & \cV_{(U_\alpha)_r}(\St_d) \cap \cV_{(U_\alpha)_{d+1}}(k) \ \subseteq \cV_{G_r}(Z_r(\lambda))\cap \cV_{(U_\alpha)_{d+1}}(k)\\
& \subseteq & \cV_{G_r}(Z_r(\lambda))\cap \cV_{G_{d+1}}(k) = \cV_{G_{d+1}}(Z_r(\lambda)).
\end{eqnarray*}
Let $L$ be the Levi subgroup of $G$ associated to the simple root $\alpha$ and let $\St_d^L$ be the $d$-th Steinberg module of $L$. Thanks to \cite[(4.2.1)]{NPV}, there is an inclusion $\cV_{L_{d+1}}(\St^L_d) \subseteq \cV_{G_{d+1}}(\St_d)$. The rank varieties of the former module were essentially computed in \cite[(7.9)]{SFB2} (The quoted result deals with Frobenius kernels of $\SL(2)$). This result implies in particular that $\cV_{(U_\alpha)_{d+1}}(\St_d^L) \ne \{0\}$, so that $\cV_{(U_\alpha)_{d+1}}(\St_d) \ne \{0\}$. Consequently,
\[ \{0\} \ne \cV_{(U_\alpha)_{d+1}}(\St_d) \cap \cV_{(U_\alpha)_{d+1}}(k) \subseteq \cV_{G_{d+1}}(Z_r(\lambda))\cap \cV_{(U_\alpha)_{d+1}}(k) = \cV_{(U_\alpha)_{d+1}}(Z_r(\lambda)).\]
As a result, the module $Z_r(\lambda)|_{(U_\alpha)_{d+1}}$ is not projective, so that  $\widehat{\ph}_\alpha(Z_r(\lambda)) \le d+1$.

Since $\Psi^d_\lambda = \Psi$, \cite[(5.6)]{FR} shows that $Z_r(\lambda)|_{G_d}$ is projective, so that module $Z_r(\lambda)|_{(U_\alpha)_d}$ is projective for every $\alpha \in
\Sigma$. It follows that $\widehat{\ph}_\alpha(Z_r(\lambda)) = d+1$ for all $\alpha \in \Sigma\setminus \Psi^{d+1}$. Since $p$ is good for $G$, \cite[(2.7)]{Ja1} ensures that the
latter set is not empty.

For $\alpha \in \Sigma\cap\Psi^{d+1}_\lambda$ we consider the Levi subgroup $L \subseteq G$ defined by $\alpha$, as well as the corresponding baby Verma module $Z_r^L(\lambda)$
of $L_r$. According to \cite[(5.6)]{FR}, the module $Z_r^L(\lambda)|_{L_{d+1}}$ is projective, and \cite[(4.2.1)]{NPV} now implies the projectivity of
$Z_r(\lambda)|_{(U_\alpha)_{d+1}}$. We conclude that $\widehat{\ph}_\alpha(Z_r(\lambda)) \ge d+2$. Consequently,
\[ \widehat{\ph}_\Sigma(Z_r(\lambda)) = \min_{\alpha \in \Sigma} \widehat{\ph}_\alpha(Z_r(\lambda)) = d+1 = \dep(\lambda),\]
as desired. \end{proof}

\bigskip

\section{Euclidean Components of finite group Schemes}\label{S:EC}
This section is concerned with the Auslander-Reiten theory of a finite group scheme $\cG$. Given a self-injective algebra $\Lambda$, we denote by $\Gamma_s(\Lambda)$ the {\it stable
Auslander--Reiten quiver} of $\Lambda$. By definition, the directed graph $\Gamma_s(\Lambda)$ has as vertices the isomorphism classes of the non-projective indecomposable
$\Lambda$-modules and its arrows are defined via the so-called {\it irreducible morphisms}. We refer the interested reader to \cite[Chap.\ VII]{ARS} for further details. The AR-quiver is fitted
with an automorphism $\tau_\Lambda$, the so-called {\it Auslander--Reiten translation}. Since $\Lambda$ is self-injective, $\tau_\Lambda$ coincides with the composite
$\Omega^2_\Lambda \circ \nu_\Lambda$ of the square of the Heller translate $\Omega_\Lambda$ and the Nakayama functor $\nu_\Lambda$, cf.\ \cite[(IV.3.7)]{ARS}.

The connected components of $\Gamma_s(\Lambda)$ are connected stable translation quivers. By work of Riedtmann \cite[Struktursatz]{Ri}, the structure of such a quiver $\Theta$ is
determined by a directed tree $T_\Theta$, and an {\it admissible group} $\Pi \subseteq {\rm Aut}_k(\ZZ[T_\Theta])$, giving rise to an isomorphism
\[ \Theta \cong \ZZ[T_\Theta]/\Pi\]
of stable translation quivers. The underlying undirected tree $\bar{T}_\Theta$, the so-called {\it tree class} of $\Theta$ is uniquely determined by $\Theta$. We refer the reader to
\cite[(4.15.6)]{Be1} for further details. For group algebras of finite groups, the possible tree classes and admissible groups were first determined by Webb \cite{We}.

Following Ringel \cite{Ri1}, an indecomposable $\Lambda$-module $M$ is called \emph{quasi-simple}, provided it lies at the end of a component of tree class $A_\infty$ of the stable
AR-quiver $\Gamma_s(\Lambda)$.

Throughout this section, $\cG$ is assumed to be a finite algebraic group, defined over an algebraically closed field $k$ of characteristic $p>0$. By general theory (see \cite[(1.5)]{FMS}), the
algebra $k\cG$ affords a Nakayama automorphism $\nu = \nu_\cG$ of finite order $\ell$. For each $n \in \NN_0$, the automorphism $\nu^n$ induces, via pull-back, an auto-equivalence of
$\modd \cG$ and hence an automorphism on the stable AR-quiver $\Gamma_s(\cG)$ of $k\cG$. In view of \cite[Chap.X]{ARS}, the Heller operator $\Omega_\cG$ also gives rise to an
automorphism of $\Gamma_s(\cG)$. Given an AR-component $\Theta \subseteq \Gamma_s(\cG)$, we denote by $\Theta^{(n)}$ the image of $\Theta$ under $\nu^n$ and put
$\Upsilon_\Theta := \bigcup_{n=0}^{\ell-1} (\Theta \cup \Omega_\cG(\Theta))^{(n)}$.

We say that a component $\Theta \subseteq \Gamma_s(\cG)$ has {\it Euclidean tree class} if the graph $\bar{T}_\Theta$ is one of the Euclidean diagrams $\tilde{A}_{12},\, \tilde{D}_n
\, (n\ge 4)$ or $\tilde{E}_n \, (6\le n\le 8)$.

\bigskip

\begin{Proposition} \label{EC1} Suppose that $\Theta \subseteq \Gamma_s(\cG)$ is a component of Euclidean tree class. Let $M$ be a $\cG$-module such that

{\rm (a)} \ $M$ possesses a filtration $(M_i)_{0\leq i \leq r}$ such that each filtration factor is indecomposable with $M_i / M_{i-1} \not \in \Upsilon_\Theta$ for $1 \leq i \leq r$, and

{\rm (b)} \ $M$ possesses a filtration $(M'_i)_{0\leq j \leq s}$ such that $M'_j/ M'_{j-1} \in \Upsilon_\Theta$ for $1 \leq j \leq s$.

\noindent
Then $\cx_\cG(M) \le 1$. \end{Proposition}

\begin{proof} Let $Z_i := M_i / M_{i-1}$ for $1 \leq i \leq r$. According to \cite[(I.8.8)]{Er}, the map
\[ d_i : \Upsilon_\Theta \lra \NN_0 \ \ ; \ \ X  \mapsto \dim_k\Ext^1_\cG(Z_i,X)\]
is additive on $\Theta^{(j)}$ and $\Omega_\cG(\Theta)^{(j)}$ for $0\leq j \leq \ell-1$. Hence \cite[(2.4)]{We} provides a natural number $b_i$ such that
\[ \dim_k \Ext^n_\cG(Z_i,X) = \dim_k \Ext^1_\cG(Z_i,\Omega_\cG^{1-n}(X)) \leq b_i\]
for every $n \geq 1$ and every $X \in \Upsilon_\Theta$. From the exactness of the sequence
\[ \Ext^n_\cG(Z_i,X) \lra \Ext^n_\cG(M_i,X) \lra \Ext^n_\cG(M_{i-1},X)\]
we obtain
\[ \dim_k\Ext^n_\cG(M,X) \leq \sum_{i=1}^rb_i =: b \ \ \ \ \ \ \forall \ X \in \Upsilon_\Theta, \ n \geq 1.\]
In view of (b), the same reasoning now implies
\[ \dim_k\Ext^n_\cG(M,M) \leq sb\ \ \ \ \ \ \forall \ n \geq 1,\]
so that general theory (cf.\ \cite[(5.3.5)]{Be2}) yields $\cx_\cG(M) \le 1$. \end{proof}

\bigskip

\begin{Corollary} \label{EC2} Suppose that $\Theta \subseteq \Gamma_s(\cG)$ is a component of Euclidean tree class. Then every $M \in \Theta$ possesses a composition factor $S$ such
that $S \in \Upsilon_\Theta$. \end{Corollary}

\begin{proof} If the composition factors of $M$ do not belong to $\Upsilon_\Theta$, then $M$ satisfies (a) and (b) of Proposition \ref{EC1}. Hence $M$ has complexity $1$, and is therefore periodic, \cite[(5.10.4)]{Be2}. As $\nu$ has finite order, there thus exists $n \ge 1$ with $\tau^n_\cG(M) \cong M$. In view of \cite[Theorem]{HPR}, the tree class
$\bar{T}_\Theta$ is either a finite Dynkin diagram or $A_\infty$, contradicting our hypothesis on $\Theta$. \end{proof}

\bigskip

\section{Representations of $\SL(2)_r$}\label{S:Rep}
Throughout, we consider the infinitesimal group scheme $\SL(2)_r$, defined over the algebraically closed field $k$ of characteristic $p \ge 3$. Recall that the algebra $\Dist(\SL(2)_r) =
k\SL(2)_r$ is symmetric (cf.\ \cite[(2.1)]{FR}), so that $\Upsilon_\Theta = \Theta \cup \Omega_{\SL(2)_r}(\Theta)$ for every  component $\Theta \subseteq \Gamma_s(\SL(2)_r)$.

The blocks of $\Dist(\SL(2)_r)$ are well understood (cf.\ \cite{Pf}). In the following, we identify the character group $X(T)$ of the standard maximal torus $T$ of diagonal matrices of
$\SL(2)$ with $\ZZ$, by letting $n \in \ZZ$ correspond to $n\rho \in X(T)$. Accordingly, the modules $L_r(0),\ldots,L_r(p^r\!-\!1)$ constitute a full set of representatives for the
isomorphism classes of the simple $\SL(2)_r$-modules (cf.\ \cite[(II.3.15)]{Ja3}). A block of $\Dist(\SL(2)_r)$ is given by a subset of $X_r(T) := \{0,\ldots,p^r\!-\!1\}$; the resulting
partition was determined by Pfautsch, see \cite[\S4.2]{Pf}: We present elements of $X_r(T)$ by expanding them $p$-adically: $\lambda = \sum_{i=0}^{r-1} \lambda_ip^i$. For $0\le i \le
\frac{p-3}{2}$ and $0\le s \le r\!-\!1$ we put
\[ \cB^{(r)}_{i,s} := \{ \sum_{i=0}^{r-1} \lambda_ip^i \ ;  \ \lambda_0=\lambda_1=\cdots =\lambda_{s-1} = p\!-\!1 \ , \ \lambda_s \in \{i,p\!-\!2\!-\!i\}\} \ \ \text{as well as} \ \
\cB_r^{(r)} = \{p^r\!-\!1\},\]
so that the corresponding block consists of those modules, whose composition factors are of the form $L_r(\lambda)$ with $\lambda \in \cB^{(r)}_{i,s}$. For future reference we record the
following result (cf.\ \cite[Satz~5]{Pf}):

\bigskip

\begin{Lemma} \label{Rep1} The functor $M \mapsto \St_s\!\otimes_k\!M^{[s]}$ induces a Morita equivalence $\cB_{i,0}^{(r-s)}\sim_M \cB_{i,s}^{(r)}$, sending $L_{r-s}(n)$ onto
$L_r(np^s\!+\!(p^s\!-\!1))$. \end{Lemma}

\begin{proof} This follows directly from \cite[(II.10.5)]{Ja3} in conjunction with Steinberg's tensor product theorem \cite[(II.3.17)]{Ja3} and \cite[(II.3.15)]{Ja3}. \end{proof}

\bigskip
\noindent
Thus, for the purposes of Auslander-Reiten theory, it suffices to consider the blocks $\cB_{i,0}^{(r)}$ for $r \ge 1$.

\bigskip

\begin{Lemma} \label{Rep2} The simple $\cB_{i,0}^{(r)}$-modules of complexity $2$ are given by the highest weights $p^r\!-\!p\!+\!i$ and $p^r\!-\!2\!-\!i$. \end{Lemma}

\begin{proof} Let $\cN(\fsl(2))$ be the nullcone of the Lie algebra $\fsl(2)$, that is,
\[ \cN(\fsl(2)) := \{ x \in \fsl(2) \ ; \ x^{[p]} = 0\}.\]
Suppose that $L_r(\lambda)$ is a simple $\cB_{i,0}^{(r)}$-module. Thanks to \cite[(7.8)]{SFB2}, we have
\[ V_r(\SL(2))_{L_r(\lambda)} \cong \{ (x_0,\ldots, x_{r-1}) \in \cN(\fsl(2))^r \ ; \ [x_i,x_j] = 0 \ \ \forall \ i,j \ , \ x_{r-i-1} = 0 \ \text{for} \ \lambda_i =p\!-\!1\}.\]
Direct computation (see \cite[p.112f]{Fa3}) shows that $\dim V_r(\SL(2))_{L_r(\lambda)} = 2$ if and only if there exists exactly one coefficient $\lambda_i \ne p\!-\!1$. As $L_r(\lambda)$
belongs to $\cB_{i,0}^{(r)}$, we conclude that $\cx_{\SL(2)_r}(L_r(\lambda)) = 2$ if and only if $\lambda_0 \in \{i,p\!-\!2\!-\!i\}$ is the only coefficient $\ne p\!-\!1$, as desired.
\end{proof}

\bigskip
\noindent
We record some structural features of the corresponding principal indecomposable $\SL(2)_r$-modules. Given $\lambda \in \{0,\ldots,p^r\!-\!1\}$, we let $P_r(\lambda)$ be the projective
cover of $L_r(\lambda)$. If $\lambda \ne p^r\!-\!1$, then $P_r(\lambda)$ is not simple, and we consider the heart
\[ \Ht_r(\lambda) = \Rad(P_r(\lambda))/\Soc(P_r(\lambda))\]
of $P_r(\lambda)$. By work of Jeyakumar \cite{Je}, each principal indecomposable $\SL(2)_1$-module $P_1(\lambda)$ has the structure of an $\SL(2)$-module, which extends the given $\SL(2)_1$-structure. Since $p\ge 3$, this structure is unique, cf.\ \cite[(II.11.11)]{Ja3}.

\bigskip

\begin{Lemma} \label{Rep3} Let $\lambda = \lambda_0 + \sum_{i=1}^{r-1}(p\!-\!1)p^i$, where $0 \le \lambda_0 \le p\!-\!2$. Then the following statements hold:

{\rm (1)} \ If $r \ge 2$, then $\Ht_r(\lambda)$ is indecomposable, with $\Soc(\Ht_r(\lambda))$ being simple.

{\rm (2)} \ The composition factors of $\Ht_r(\lambda)$ are of the form $L_r(\mu)$, with
\[\mu \in \{p\!-\!2\!-\!\lambda_0+(p\!-\!2)p^{\ell\!-\!1}+\sum_{i=\ell}^{r-1}(p\!-\!1)p^i \ ; \ 1< \ell \le r\} \cup \{p\!-\!2\!-\!\lambda_0\}.\] \end{Lemma}

\begin{proof} (1) Thanks to \cite[(1.1)]{HJ}, the principal indecomposable $\SL(2)_r$-module with highest weight $\lambda$ has the form
\[ P_r(\lambda) = P_1(\lambda_0)\!\otimes_k\! L_1(p\!-\!1)^{[1]}\!\otimes_k\cdots\otimes_k\!L_1(p\!-\!1)^{[r-1]}.\]
As shown in \cite[Thm.3]{Hu2}, the $\SL(2)$-module $P_1(\lambda_0)$ has composition factors (from top to bottom) $L(\lambda_0),L(2p\!-\!2\!-\! \lambda_0),L(\lambda_0)$, where
$L(\gamma)$ denotes the simple $\SL(2)$-module with highest weight $\gamma \in \NN_0$. From the isomorphism $L(\lambda_0)|_{\SL(2)_1} \cong L_1(\lambda_0)$ (cf.\
 \cite[(II.3.15)]{Ja3}), we obtain
\[ \Ht_r(\lambda) \cong L(2p\!-\!2\!-\!\lambda_0)|_{\SL(2)_r}\!\otimes_k\! L_1(p\!-\!1)^{[1]}\!\otimes_k\cdots\otimes_k\!L_1(p\!-\!1)^{[r-1]}.\]
Steinberg's tensor product theorem in conjunction with \cite[(II.3.15)]{Ja3}, \cite[(3.1)]{AJL} and $r\ge 2$ now yields
\begin{eqnarray*}
\Ht_r(\lambda) & \cong &  L_1(p\!-\!2\!-\!\lambda_0)\!\otimes_k\!L_1(1)^{[1]}\!\otimes_k\!L_1(p\!-\!1)^{[1]}\!\otimes_k\cdots\otimes_k\!L_1(p\!-\!1)^{[r-1]} \\
& \cong & L_1(p\!-\!2\!-\!\lambda_0)\!\otimes_k\!P_1(p\!-\!2)^{[1]}\!\otimes_k\!L_1(p\!-\!1)^{[2]}\!\otimes_k \cdots\otimes_k\!L_1(p\!-\!1)^{[r-1]}.
\end{eqnarray*}
Since the latter module is contained in the principal indecomposable $\SL(2)_r$-module
\[ P_1(p\!-\!2\!-\!\lambda_0)\!\otimes_k\!P_1(p\!-\!2)^{[1]}\!\otimes_k\!L_1(p\!-\!1)^{[2]}\!\otimes_k \cdots\otimes_k\!L_1(p\!-\!1)^{[r-1]}\]
(see \cite[(1.1)]{HJ}), we conclude that $\Ht_r(\lambda)$ is indecomposable with simple socle.

(2) This is a direct consequence of \cite[(3.4)]{AJL}, \cite[(II.3.15)]{Ja3} and Steinberg's tensor product theorem. \end{proof}

\bigskip
\noindent
We turn to the Auslander-Reiten theory of $\SL(2)_r$ and begin by determining the stable AR-components of Euclidean tree class. We recall that the standard almost split sequence
\[ (0) \lra \Rad(P_r(\lambda)) \lra \Ht_r(\lambda)\oplus P_r(\lambda) \lra P_r(\lambda)/\Soc(P_r(\lambda)) \lra (0)\]
is the only almost split sequence involving the principal indecomposable module $P_r(\lambda)$ (cf.\ \cite[(V.5.5)]{ARS}).

\bigskip

\begin{Proposition} \label{Rep4} Let $\Theta \subseteq \Gamma_s(\SL(2)_r)$ be a component of Euclidean tree class. Then $\Theta \cong \ZZ[\tilde{A}_{12}]$. \end{Proposition}

\begin{proof} In view of Lemma \ref{Rep1}, we may assume that $\Theta \subseteq \Gamma_s(\cB_{i,0}^{(r)})$. According to \cite[(2.4)]{We}, the component $\Theta$ is attached to a
principal indecomposable module, so by passing to the isomorphic component $\Omega_{\SL(2)_r}^{-1}(\Theta)$ we may assume without loss of generality that $\Theta$ contains a simple
module $L_r(\lambda)$.

Another application of \cite[(2.4)]{We} implies $\cx_{\SL(2)_r}(M) = 2$ for every $M \in \Upsilon_\Theta$, so that Lemma \ref{Rep2} yields
\[ \lambda = \lambda_0 + (p\!-\!1)p + \cdots +(p\!-\!1)p^{r-1} \ \ ; \ \ \lambda_0 \in \{i,p\!-\!2\!-\!i\}.\]
Suppose that $r \ge 2$. By virtue of Lemma \ref{Rep3}(1), the module $\Ht_r(\lambda)$ is indecomposable and $\Ht_r(\lambda) \in \Omega_{\SL(2)_r}(\Theta) \subseteq
\Upsilon_\Theta$.

Owing to Lemma \ref{Rep3}(2) and Lemma \ref{Rep2} the composition factors of $\Ht_r(\lambda)$ have complexity $\ne 2$. Accordingly, they do not belong to $\Upsilon_\Theta$, and
Corollary \ref{EC2} yields a contradiction. Consequently, $r=1$, and the assertion follows from the well-known path algebra presentation of $\cB_{i,0}^{(1)}$ (see \cite{Dr,Fi,Ru}), which
establishes a Morita equivalence between $\cB_{i,0}^{(1)}$ and the trivial extension of the Kronecker algebra $k[\bullet \rightrightarrows \bullet]$. \end{proof}

\bigskip

\begin{Remark} The proof of the foregoing result also shows that $\Theta$ belongs to a block of $k\SL(2)_r$ of tame representation type (see also Theorem \ref{BF3} below). \end{Remark}

\bigskip
\noindent
Recall that a non-projective indecomposable $\SL(2)_r$-module $M$ is referred to as quasi-simple, provided

(a) \ the module $M$ belongs to a component of tree class $A_\infty$, and

(b) \ $M$ has exactly one predecessor in $\Gamma_s(\SL(2)_r)$.

\bigskip

\begin{Proposition} \label{Rep5} Let $S$ be a simple $\SL(2)_r$-module of complexity $\cx_{\SL(2)_r}(S)=2$, $\Theta \subseteq \Gamma_s(\SL(2)_r)$ be the component containing $S$.
Then $\Theta \cong \ZZ[A_\infty], \ \ZZ[\tilde{A}_{12}]$. If $\Theta \cong \ZZ[A_\infty]$, then $S$ is quasi-simple. \end{Proposition}

\begin{proof} Since Morita equivalence preserves the complexity of a module, Lemma \ref{Rep1} allows us to assume that $\Theta \subseteq \Gamma_s(\cB_{i,0}^{(r)})$. As
$\cx_{\SL(2)_r}(S) = 2$, the component $\Theta$ is not finite and its tree class is not a finite Dynkin diagram (cf.\ \cite[(2.1)]{Fa3}). If the tree class $\bar{T}_\Theta$ is Euclidean, then
Proposition \ref{Rep4} implies $\Theta \cong \ZZ[\tilde{A}_{12}]$. Alternatively, $\bar{T}_\Theta \cong A_\infty, A_\infty^\infty$ or $D_\infty$ (see \cite[(1.3)]{Fa3}).

In view of Lemma \ref{Rep2}, we have $S \cong L_r(\lambda)$, where
\[ \lambda = \lambda_0 + (p\!-\!1)p + \cdots +(p\!-\!1)p^{r-1} \ \ ; \ \ \lambda_0 \in \{i,p\!-\!2\!-\!i\}.\]
If $r=1$, then $\Theta \cong \ZZ[\tilde{A}_{12}]$, a contradiction. Alternatively, Lemma \ref{Rep3} shows that $\Ht_r(\lambda)$ is indecomposable with simple socle. From the standard
almost split sequence involving $P_r(\lambda)$ we conclude that $P_r(\lambda)/\Soc(P_r(\lambda)) \cong \Omega^{-1}_{\SL(2)_r}(L_r(\lambda))$ has exactly one predecessor in
$\Gamma_s(\SL(2)_r)$. Since $\Omega_{\SL(2)_r}$ defines an automorphism of $\Gamma_s(\SL(2)_r)$, the simple module $S=L_r(\lambda)$ enjoys the same property. Thus,
$\bar{T}_\Theta \not \cong A_\infty^\infty$, and if $\bar{T}_\Theta \cong A_\infty$, then $S$ is quasi-simple.

It thus remains to rule out the case, where $\bar{T}_\Theta \cong D_\infty$. Assuming this isomorphism to hold, we infer from the above that $\Ht_r(\lambda)$ has $3$ successors in $\Psi
:= \Omega^{-1}_{\SL(2)_r}(\Theta)$. Accordingly, there exists an almost split sequence
\[ (0) \lra X \lra \Ht_r(\lambda) \oplus Q \lra Y \lra (0),\]
with $Y\not \cong P_r(\lambda)/L_r(\lambda)$, and with $Q \not \cong P_r(\lambda)$ being indecomposable projective, or zero. If $Q \ne (0)$, then $Q = P_r(\lambda')$ belongs to the
block of $L_r(\lambda)$ and the standard sequence yields $\Ht_r(\lambda) \cong \Ht_r(\lambda')$. Since $\bar{T}_\Theta \cong D_\infty$, \cite[(2.2)]{Fa3} implies
$\cx_{\SL(2)_r}(L_r(\lambda')) =2$, and Lemma \ref{Rep2} yields
\[ \lambda' = p-2-\lambda_0 + (p\!-\!1)p + \cdots +(p\!-\!1)p^{r-1}.\]
This, however, contradicts (2) of Lemma \ref{Rep3}.

We thus have an almost split sequence
\[ (0) \lra X \lra \Ht_r(\lambda) \lra Y \lra (0).\]
Since $\Soc(\Ht_r(\lambda))$ is simple, we may apply \cite[(V.3.2)]{ARS} to conclude that $Y$ is simple. As $\cx_{\SL(2)_r}(Y) = 2$ (cf.\ \cite[(2.2)]{Fa3}), Lemma \ref{Rep2} implies
$Y \in \{L_r(\lambda),L_r(\lambda')\}$, so that Lemma \ref{Rep3}(2) again leads to a contradiction.

Consequently, $\bar{T}_\Theta = A_\infty$, which, in view of $\cx_{\SL(2)_r}(L_r(\lambda)) = 2$, implies $\Theta \cong \ZZ[A_\infty]$ (see \cite[(2.1)]{Fa3}). \end{proof}

\bigskip

\section{Blocks and AR-Components}\label{S:BAR}
In this section we illustrate our results by considering blocks and Auslander-Reiten components of infinitesimal group schemes.

\subsection{Invariants of AR-Components}
We begin by showing that the notion of projective height gives rise to invariants of stable Auslander-Reiten components. Let $\fE \subseteq \fE(\cG)$ be a collection of elementary abelian
subgroups of $\cG$, and define
\[ \ph_\fE(M) := \max_{\cU \in \fE}\ph_\cU(M).\]

\bigskip

\begin{Prop} \label{NI1} Let $\cG$ be an infinitesimal group, $\Theta \subseteq \Gamma_s(\cG)$ be a connected component of its stable Auslander-Reiten quiver. Given $\fE
\subseteq \fE(\cG)$, we have
\[ \ph_\fE(M) = \ph_\fE(N)\]
for all $M,N \in \Theta$. \end{Prop}

\begin{proof} Let $\cU \in \fE$ be an elementary abelian subgroup and consider $M \in \Theta$. We write $t := \ph_\cU(M)$. According to \cite[(1.5)]{FMS}, the Nakayama functor
$\nu_\cG$ is induced by the automorphism $\id_\cG\ast \zeta_\ell$, which is the convolution of $\id_\cG$ with the left modular function $\zeta_\ell : k\cG \lra k$. As $\cU_s$ is unipotent,
the restriction $\zeta_\ell|_{k\cU_s}$ of $\zeta_\ell$ coincides with with the co-unit, so that $(\id_\cG\ast \zeta_\ell)|_{k\cU_s} = \id_{k\cU_s}$. It follows that $\nu_\cG(M)|_{\cU_s} =
M|_{\cU_s}$ for all $s$.

Since $\tau_\cG = \Omega^2_\cG \circ \nu_\cG$, we therefore obtain
\[ \tau_\cG(M)|_{\cU_s} \cong \Omega^2_{U_s}(M|_{\cU_s})\oplus ({\rm proj.}) \ \ \ \ \forall \ s \in \NN.\]
Thus, $\tau_\cG(M)|_{U_s}$ is projective if and only if $M|_{\cU_s}$ is projective. This implies $t = \ph_\cU(\tau_\cG(M))$.

Now let
\[ \fE_M : \ \ \ \ \ \ \ \ \ \ (0) \lra \tau_\cG(M) \lra \bigoplus_{i=1}^n E_i \lra M \lra (0)\]
be the almost split sequence terminating in $M$. If $s \le t-1$, then $M|_{\cU_s}$ is projective and $\fE_M|_{\cU_s}$ splits, so that each $E_i|_{\cU_s}$ is a direct summand of the
projective module $M|_{\cU_s}\oplus \tau_\cG(M)|_{\cU_s}$. Hence $E_i|_{\cU_s}$ is projective and $\ph_\cU(E_i) \ge t$ for $1\le i \le n$.

Let $M,N \in \Theta$. By the above, we have $\ph_\cU(N) \ge \ph_\cU(M)$, whenever $N$ is a predecessor of $M$. In that case, there also exists an arrow $\tau_\cG(M) \rightarrow N$, so
that $\ph_\cU(M) = \ph_\cU(\tau_\cG(M)) \ge \ph_\cU(N)$, whence $\ph_\cU(M) = \ph_\cU(N)$.

Since $\Theta$ is connected, it follows that $\ph_\cU(M) = \ph_\cU(N)$ for all $M,N \in \Theta$. This readily implies our assertion. \end{proof}

\bigskip

\begin{Cor} \label{NI2} Suppose that $\Theta \subseteq \Gamma_s(\cG)$ is a component containing a module of complexity $1$. Then there exists an elementary abelian subgroup
$\cU_\Theta \subseteq \cG$ such that $\ph(M) = \height(\cU_\Theta)$ for every $M \in \Theta$. \end{Cor}

\begin{proof} Let $M,N \in \Theta$. As observed in \cite[(1.1)]{Fa3}, we have
\[ V_r(\cG)_M = V_r(\cG)_N,\]
so that the subgroups $\cU_M$ and $\cU_N$, defined in (\ref{PH3}), coincide. Setting $\cU_\Theta := \cU_M$ for some $M \in \Theta$, we obtain the desired result directly from
Theorem \ref{PH3}. \end{proof}

\subsection{Blocks of finite representation type}
Throughout, we consider an infinitesimal group $\cG$ of height $r$ and denote by $\nu_\cG : \modd  \cG \lra \modd \cG$ the Nakayama functor of $\modd \cG$. By general theory
\cite[(1.5)]{FMS}, $\nu_\cG$ is an auto-equivalence of order a power of $p$.

\bigskip

\begin{Cor} Let $\cG$ be an infinitesimal group, $\cB \subseteq k\cG$ be a block of finite representation type.

{\rm (1)} \  If $\cU �\in \fE(\cG)$ is an elementary abelian subgroup, then we have $\ph_\cU(M) = \ph_\cU(N)$ for any two non-projective indecomposable $\cB$-modules $M$ and $N$.

{\rm (2)} \ There exists an elementary abelian subgroup $\cU_\cB \subseteq \cG$ such that $\ph(M) = \height(\cU_\cB)$ for every non-projective indecomposable $\cB$-module $M$.
\end{Cor}

\begin{proof} Auslander's Theorem \cite[(VI.1.4)]{ARS} implies that $\Gamma_s(\cB)$ is a connected component of $\Gamma_s(\cG)$. Hence the assertions follow directly from Proposition
\ref{NI1} and Corollary \ref{NI2}. \end{proof}

\bigskip
\noindent
The group $\cU_\cB$ may be thought of as a ``defect group'' of the block $\cB$.  Recall that a finite-dimensional  $k$-algebra $\Lambda$ is called a {\it Nakayama algebra} if all of its
indecomposable projective left modules and indecomposable injective left modules are uniserial (i.e., they only possess one composition series).

The following result provides a first link between properties of $\cU_\cB$ and $\cB$:

\bigskip

\begin{Prop} \label{BFR1} Let $S$ be a simple $\cG$-module such that $\cx_\cG(S) = 1$ and $\ph(S) =r$. Then the block $\cB \subseteq k\cG$ containing $S$ is a Nakayama algebra with simple modules $\{\nu_\cG^i(S) \ ; \ i \in \NN_0\}$. \end{Prop}

\begin{proof} Since $S$ has complexity $1$ with $\ph(S) = r$, Corollary \ref{PH2} implies $\Omega_\cG^2(S) \cong S$. The arguments employed in the proof of \cite[(3.2)]{Fa1}
now yield the assertion. \end{proof}

\bigskip
\begin{Example} Consider the subgroup $\cG := \SL(2)_1\GG_{a(2)} \subseteq \SL(2)$, given by
\[ \cG(R) = \{ \left( \begin{array}{cc} a & b \\ c & d \end{array} \right) \in \SL(2)(R) \ ; \ a^p=1=d^p \ , c^p = 0 \ , \ b^{p^2} = 0\}\]
for every commutative $k$-algebra $R$. The first Steinberg module $\St_1$ is a simple $\cG$-module, whose restriction to $\cG_1 = \SL(2)_1$ is simple and projective, whence $\ph(\St_1) = 2$. Since $\cG/\cG_1 \cong \GG_{a(1)}$, \cite[(I.6.6)]{Ja3} implies
\[ \Ext^n_\cG(\St_1,\St_1) \cong \HH^n(\GG_{a(1)},\Hom_{\cG_1}(\St_1,\St_1)) \cong  \HH^n(\GG_{a(1)},k),\]
so that $\cx_\cG(\St_1) = 1$, cf.\ \cite[\S2]{Ca1}. Thus, the block $\cB \subseteq k\cG$ containing $\St_1$ is a Nakayama algebra, with $\St_1$ being the only simple $\cB$-module. \end{Example}

\subsection{Frobenius kernels of reductive groups}
In this section we consider a smooth reductive group $G$ of characteristic $p\ge 3$. Our objective is to apply the foregoing results in order to correct the proof of \cite[(4.1)]{Fa3}. Given $r
\ge 1$, we are interested in the algebra $\Dist(G_r) = kG_r$ of distributions of $G_r$.

The following Lemma, which follows directly from the arguments of \cite[(7.2)]{Fa2}, reduces a number of issues to the special case $G = \SL(2)$:

\bigskip

\begin{Lem} \label{BF1} Let $\cB \subseteq \Dist(G_r)$ be a block. If $\cB$ has a simple module of complexity $2$, then $\cB$ is Morita equvalent to a block of $\Dist(\SL(2)_r)$. \hfill
$\square$ \end{Lem}

\bigskip
\noindent
The proof of the following result corrects the false reference to $\SL(2)$-theory on page 113 of \cite{Fa3}. Theorem 4.1 of \cite{Fa3} is correct as stated.

\bigskip

\begin{Thm} \label{BF2} A non-projective simple $G_r$-module $L_r(\lambda)$ of complexity $2$ belongs to a component $\Theta \cong \ZZ[A_\infty],\, \ZZ[\tilde{A}_{12}]$.\end{Thm}

\begin{proof} This follows from a consecutive application of Lemma \ref{BF1} and Proposition \ref{Rep5}. \end{proof}

\bigskip
\noindent
A finite-dimensional $k$-algebra $\Lambda$ is called {\it tame} if it is not of finite representation type and if for every $d>0$ the $d$-dimensional indecomposable $\Lambda$-modules can by parametrized by a one-dimensional variety. The reader is referred to \cite[(I.4)]{Er} for the precise definition.

The structure of the representation-finite and tame blocks of the Frobenius kernels of smooth groups is well understood, see \cite[Theorem]{Fa4}. The following result shows that, for smooth
reductive groups, such blocks may be characterized via the complexities of their simple modules.

\bigskip

\begin{Thm} \label{BF3} Suppose that $G$ is reductive, and let $\cB \subseteq \Dist(G_r)$ be a block. Then the following statements are equivalent:

{\rm (1)} \ $\cB$ is tame.

{\rm (2)} \ Every simple $\cB$-module has complexity $2$.

{\rm (3)} \ $\Gamma_s(\cB)$ possesses a component isomorphic to $\ZZ[\tilde{A}_{12}]$.\end{Thm}

\begin{proof} (1) $\Rightarrow$ (2) Passing to the connected component of $G$ if necessary, we may assume that $G$ is connected. Let $S$ be a simple $\cB$-module. Since $\cB$ is tame,
\cite[(3.2)]{Fa6} implies $\cx_{G_r}(S) \le 2$. Since $\cB$ is not representation finite, the simple $\cB$-module $S$ is not projective, so that $\cx_{G_r}(S) \ge 1$. Suppose that
$\cx_{G_r}(S) = 1$. Since $S$ is $G$-stable, Corollary \ref{PH3.5} shows that $\cU_\cB \unlhd \cG_r$ is a unipotent, normal subgroup of $G_r$. Passage to the first Frobenius kernels yields
the existence of a non-zero unipotent $p$-ideal $\fu \unlhd \fg$. Since $G$ is reductive, \cite[(11.8)]{Hu1} rules out the existence of such ideals, a contradiction.

(2) $\Rightarrow$ (3) By Lemma \ref{BF1}, the block $\cB$ is Morita equivalent to a block $\cB' \subseteq \Dist(\SL(2)_r)$, all whose simple modules have complexity $2$. A consecutive
application of Lemma \ref{Rep2} and Lemma \ref{Rep3} implies $r=1$. Consequently, $\cB'$ is tame and so is $\cB$.

(3) $\Rightarrow$ (1) Suppose that $\Gamma_s(\cB)$ possesses a component of type $\ZZ[\tilde{A}_{12}]$. By \cite[(2.1)]{FR}, the block $\cB$ is symmetric, and we may invoke \cite[(IV.3.8.3),(IV.3.9)]{Er} to see that the algebra $\cB/\Soc(\cB)$ is special biserial. Thus, $\cB/\Soc(\cB)$ is tame or representation-finite (cf.\ \cite[(II.3.1)]{Er}) and, having the
same non-projective indecomposables, $\cB$ enjoys the same property. Since $\cB$ possesses modules of complexity $2$, it is not of finite representation type, cf.\ \cite{He}. \end{proof}

\bigskip

\begin{Remark} Let $G$ be a smooth algebraic group scheme. According to \cite[(4.6)]{Fa4}, the presence of a tame block $\cB \subseteq \Dist(G_r)$ implies that $G$ is reductive.
\end{Remark}

\bigskip

\section{The Nakayama Functor of $\modd \gr\Lambda$}\label{S:NF}
In preparation for our analysis of $G_rT$-modules we study in this section the category of graded modules of an associative algebra. Let $k$ be a field,
\[ \Lambda  = \bigoplus_{i \in \ZZ^n} \Lambda_i\]
be a finite-dimensional, $\ZZ^n$-graded $k$-algebra. We denote by $\modd \gr \Lambda$ the category of finite-dimensional $\ZZ^n$-graded $\Lambda$-modules and degree zero homomorphisms. Given $i \in \ZZ^n$, the $i$-th shift functor $[i] : \modd \gr\Lambda \lra \modd \gr\Lambda$ sends $M$ onto $M[i]$, where
\[ M[i]_j := M_{j-i} \ \ \ \ \ \forall \ j \in \ZZ\]
and leaves the morphisms unchanged. The $\Lambda^{\rm op}$-module $\Hom_\Lambda(M,\Lambda)^\ast$ has a natural $\ZZ^n$-grading. There results a functor
\[ \cN : \modd \gr \Lambda \lra \modd \gr\Lambda \ \ ; \ \ M \mapsto \Hom_\Lambda(M,\Lambda)^\ast,\]
the {\it Nakayama functor} of $\modd \gr\Lambda$. The purpose of this section is to determine this functor for certain Hopf algebras $\Lambda$. We begin with a few general observations.

\subsection{Graded Frobenius Algebras} Suppose that $\Lambda$ is a Frobenius algebra with Frobenius homomorphism $\pi : \Lambda \lra k$ and associated non-degenerate bilinear form
\[ (\, ,\, )_\pi : \Lambda\times \Lambda \lra k \ \ ; \ \ (x,y) \mapsto \pi(xy).\]
The uniquely determined automorphism $\mu : \Lambda \lra \Lambda$ given by
\[ (y,x)_\pi = (\mu(x),y)_\pi\]
is called the {\it Nakayama automorphism} of the form $(\, , \,)_\pi$. For an arbitrary automorphism $\alpha \in {\rm Aut}_k(\Lambda)$ and a $\Lambda$-module $M$, we denote by
$M^{(\alpha)}$ the $\Lambda$-module with underlying $k$-space $M$ and action
\[ a\dact m := \alpha^{-1}(a)m \ \ \ \ \ \ \forall \ a \in \Lambda,\, m \in M.\]
Given $M \in \modd \gr \Lambda$, we define the {\it support of} $M$ via
\[ \supp(M) := \{ i \in \ZZ^n \ ; \ M_i \ne (0)\}.\]
Then we have $\supp(M[d]) = d+\supp(M)$, as well as $\supp(M^\ast) = -\supp(M)$. If $\alpha \in {\rm Aut}_k(\Lambda)$ is an automorphism of degree $0$ and $M \in \modd \gr\Lambda$ is graded, then $M^{(\alpha)}$ is graded via
\[ M^{(\alpha)}_i := M_i \ \ \ \ \ \ \forall \ i \in \ZZ^n.\]
In particular, $\supp(M) = \supp(M^{(\alpha)})$. Recall that the {\it enveloping algebra} $\Lambda^e := \Lambda\!\otimes_k\!\Lambda^{\rm op}$ inherits a $\ZZ^n$-grading from
$\Lambda$ via
\[ \Lambda^e_i := \bigoplus_{\ell+j = i} \Lambda_\ell\!\otimes_k\!\Lambda_j \ \ \ \ \forall \ i \in \ZZ^n.\]
The algebra $\Lambda$ obtains the structure of a graded $\Lambda^e$-module by means of
\[ (a\otimes b)\dact c := acb \ \ \ \ \forall \ a,b,c \in \Lambda.\]
Our first result extends the classical formula for the Nakayama functor of Frobenius algebras to the graded case.

\bigskip

\begin{Lem} \label{NF1} Suppose that $\Lambda$ is a Frobenius algebra with Frobenius homomorphism $\pi : \Lambda \lra k$ of degree $d_\Lambda$. Then the following statements
hold:

{\rm (1)} \ The Nakayama automorphism $\mu : \Lambda \lra \Lambda$ of the form $(\, ,\,)_\pi$ has degree $0$.

{\rm (2)} \ There are natural isomorphisms
\[ \cN(M) \cong M^{(\mu^{-1})}[d_\Lambda]\]
for all $M \in \modd \gr \Lambda$.

{\rm (3)} \ Every homogeneous Frobenius homomorphism of $\Lambda$ has degree $d_\Lambda$.\end{Lem}

\begin{proof} (1) Since $\pi$ has degree $d_\Lambda$, we have $\pi(\Lambda_i) = (0)$ for $i\ne -d_\Lambda$, whence
\[ (\Lambda_i,\Lambda_j)_\pi = (0) \ \ \ \ \text{for} \ i+j \ne -d_\Lambda.\]
Let $a \in \Lambda_i$ and write $\mu(a)= \sum_{j \in \ZZ^n} x_j$. Assuming $j \ne i$, we consider $b \in \Lambda_\ell$. Then we have $(x_j,b)_\pi = 0$ for $j+\ell \ne -d_\Lambda$. If
$j+\ell = -d_\Lambda$, we obtain, observing $i+\ell \ne -d_\Lambda$,
\[ (x_j,b)_\pi = (\mu(a),b)_\pi = (b,a)_\pi  = 0.\]
As a result, $x_j = 0$ whenever $j\ne i$, so that $\mu(a) = x_i \in \Lambda_i$.

(2) Let $\gamma := \mu\otimes \id_\Lambda$ be the induced automorphism of the enveloping algebra $\Lambda^e$. Since $\pi$ has degree $d_\Lambda$, the map
\[ \Psi : \Lambda^{(\gamma^{-1})} \lra \Lambda^\ast \ \ ; \ \ \Psi(x)(y) := (x,y)_\pi\]
is an isomorphism of $\Lambda^e$-modules of degree $d_\Lambda$. There results an isomorphism
\[ \Lambda^{(\gamma^{-1})}[d_\Lambda] \cong \Lambda^\ast\]
of graded $\Lambda^e$-modules. The adjoint isomorphism theorem yields natural isomorphisms
\[ \cN(M) \cong \Lambda^\ast\!\otimes_\Lambda\!M \cong \Lambda^{(\gamma^{-1})}[d_\Lambda]\!\otimes_\Lambda\!M \cong M^{(\mu^{-1})}[d_\Lambda],\]
for every $M \in \modd \gr \Lambda$, cf.\ \cite[(III.2.9)]{ASS}.

(3) If $\Lambda$ has a Frobenius homomorphism of degree $d'$, then (2) provides an isomorphism $\Lambda^{(\delta^{-1})}[d']$ $ \cong \Lambda^\ast$, where $\delta = \nu\otimes
\id_\Lambda$ for some Nakayama automorphism $\nu$ of $\Lambda$ of degree $0$. We therefore obtain
\[ \supp(\Lambda)+d_\Lambda = \supp(\Lambda^{(\gamma^{-1})})+d_\Lambda = \supp(\Lambda^{(\gamma^{-1})}[d_\Lambda]) = -\supp(\Lambda) = \supp(\Lambda)+d',\]
so that $x \mapsto x+(d_\Lambda-d')$ leaves $\supp(\Lambda)$ invariant. Since $\supp(\Lambda)$ is finite, it follows that $d'=d_\Lambda$. \end{proof}

\bigskip

\begin{Lem} \label{NF3} Let $\Lambda = \bigoplus_{i \in \ZZ^n}\Lambda_i$ be a graded $k$-algebra, $\cB \subseteq \Lambda$ be a block of $\Lambda$. Then the following statements
hold:

{\rm (1)} \ Any central idempotent of $\Lambda$ is homogeneous of degree $0$.

{\rm (2)} \ The block $\cB \subseteq \Lambda$ is a homogeneous subspace.

{\rm (3)} \ Suppose that $\Lambda$ is a Frobenius algebra with homogeneous Frobenius homomorphism of degree $d_\Lambda$. Then $\cB$ is a Frobenius algebra, and every homogeneous
Frobenius homomorphism of $\cB$ has degree $d_\Lambda$. \end{Lem}

\begin{proof} (1) General theory provides a torus $T$, which acts on $\Lambda$ via automorphisms such that the given grading coincides with the weight space decomposition $\Lambda =
\bigoplus_{\lambda \in X(T)} \Lambda_\lambda$ of $\Lambda$ relative to $T$. Let $\cI\subseteq \Lambda$ be the set of central primitive idempotents of $\Lambda$. Then $\cI$ is a finite,
$T$-invariant set, so that the connected algebraic group $T$ acts trivially on $\cI$. Hence $\cI \subseteq \Lambda_0$, and our assertion follows from the fact that every central idempotent is a
sum of elements of $\cI$.

(2) There exists a central primitive idempotent $e$ such that $\cB = \Lambda e.$ Thanks to (1), we have $e \in \Lambda_0$, whence
\[ \cB = \bigoplus_{i \in \ZZ^n}\Lambda_ie = \bigoplus_{i \in \ZZ^n}(\cB\cap \Lambda_i),\]
as desired.

(3) Let $\pi : \Lambda \lra k$ be a Frobenius homomorphism of degree $d_\Lambda$. If $I \subseteq \cB$ is a left  ideal of $\cB$ which is contained in $\ker \pi|_{\cB}$, then $I$ is also a
left ideal of $\Lambda$, so that $I = (0)$. In view of (2), the block $\cB$ is a homogeneous subspace. We conclude that $\pi|_\cB$ is a homogeneous Frobenius homomorphism of $\cB$ of
degree $d_\Lambda$. Our assertion now follows from Lemma \ref{NF1}. \end{proof}

\bigskip
\noindent
A $k$-algebra $\Lambda$ is called {\it connected} if its $\Ext$-quiver is connected. This is equivalent to $\Lambda$ having exactly one block. The following result in conjunction with the
observations above shows when graded Frobenius algebras afford homogeneous Frobenius homomorphisms.

\bigskip

\begin{Thm} \label{NF2} Let $\Lambda = \bigoplus_{i \in \ZZ^n}\Lambda_i$ be a connected graded Frobenius algebra. Then there exists a homogeneous Frobenius homomorphism $\pi :
\Lambda \lra k$. \end{Thm}

\begin{proof} (1) Let $\Lambda^e = \Lambda\!\otimes_k\!\Lambda^{\rm op}$ be the enveloping algebra of $\Lambda$. As $\Lambda$ is connected, the canonical $\Lambda^e$-module
$\Lambda$ is indecomposable. Let $\pi : \Lambda \lra k$ be a Frobenius homomorphism. As argued above, there exists a degree $0$ automorphism $\gamma \in {\rm Aut}_k(\Lambda^e)$
such that the map
\[ \Psi : \Lambda^{(\gamma^{-1})} \lra \Lambda^\ast \ \ ; \ \ \Psi(x)(y) := (x,y)_\pi\]
is an isomorphism of $\Lambda^e$-modules. As all modules involved are indecomposable, \cite[(4.1)]{GG1} ensures the existence of an element $d \in \ZZ^n$ and an isomorphism
\[ \Phi : \Lambda^{(\gamma^{-1})}[d] \lra \Lambda^\ast\]
of graded $\Lambda^e$-modules. Consequently, the map
\[ \rho : \Lambda \lra k \ \ ; \ \ \rho(x) := \Phi(1)(x)\]
is a Frobenius homomorphism of degree $d$. \end{proof}

\bigskip

\begin{Remark} Let $\Lambda$ and $\Gamma$ be two connected $\ZZ^n$-graded Frobenius algebras with homogeneous Frobenius homomorphisms of degrees $d_\Lambda$ and
$d_\Gamma$, respectively. If $d_\Lambda \ne d_\Gamma$, then Lemmas \ref{NF1} and \ref{NF3} imply that $\Lambda \oplus \Gamma$ is a graded Frobenius algebra, which does not
afford a homogeneous Frobenius homomorphism. We shall see in the next section that such phenomena do not arise within the context of graded Hopf algebras. \end{Remark}

\bigskip
\noindent
Suppose that $\Lambda$ is a Frobenius algebra. Then $\modd \gr \Lambda$ is a Frobenius category, and \cite[(3.5)]{GG2} ensures that $\modd \gr \Lambda$ has almost split sequences.
In view of \cite[\S1]{GG2}, the Auslander-Reiten translation $\tau_{\gr \Lambda}$ is given by
\[ \tau_{\gr\Lambda} = \cN \circ \Omega^2_{\gr \Lambda},\]
where $\Omega_{\gr\Lambda}$ denotes the Heller operator of the Frobenius category $\modd \gr \Lambda$. We denote by $\Gamma_s(\gr\Lambda)$ the stable Auslander-Reiten quiver of
$\modd \gr \Lambda$.

\bigskip

\begin{Cor} \label{NF4} Let $\Theta \subseteq \Gamma_s(\gr \Lambda)$ be a component. Then there exists an automorphism $\mu$ of $\Lambda$ of degree $0$ and
$d_\Theta \in \ZZ$ such that
\[ \tau_{\gr\Lambda}(M) = \Omega^2_{\gr\Lambda}(M^{(\mu)})[d_\Theta]\]
for every $M \in \Theta$. \end{Cor}

\begin{proof} Let
\[ \Lambda = \cB_1\oplus \cB_2 \oplus \cdots \oplus \cB_s\]
be the block decomposition of $\Lambda$. In view of Lemma \ref{NF3}, this is a decomposition of graded Frobenius algebras which gives rise to a decomposition
\[ \modd \gr\Lambda = \bigoplus_{i=1}^s\modd \gr\cB_i\]
of the graded module category.

Suppose that $M \in \Theta$. Since $M$ is indecomposable, there exists a block $\cB_\ell \subseteq \Lambda$ such that $M \in \modd \gr \cB_\ell$. This readily implies that $\Theta
\subseteq \modd \gr\cB_\ell$, whence $\Theta \subseteq \Gamma_s(\gr\cB_\ell)$. Our assertion now follows from Theorem \ref{NF2}.\end{proof}

\bigskip

\subsection{Graded Hopf Algebras}
Suppose that $H$ is a finite-dimensional Hopf algebra. We say that $H$ is {\it graded} if $H = \bigoplus_{i\in \ZZ^n}H_i$ is a graded $k$-algebra such that the comultiplication $\Delta : H
\lra H\!\otimes_k\!H$ is homogeneous of degree $0$. In that case, the counit $\varepsilon : H \lra k$ and the antipode $\eta : H \lra H$ are also maps of degree $0$. The subspace
\[ \int^\ell_H := \{ x \in H \ ; \ hx = \varepsilon(h)x \ \ \forall \ h \in H\}\]
is called the space of {\it left integrals} of $H$. This space is known to be one-dimensional, cf.\ \cite{Sw}.

\bigskip

\begin{Lem} \label{HA1} Let $H$ be a graded Hopf algebra. Then the following statements hold:

{\rm (1)} \ There exists $i \in \ZZ^n$ such that $\int_H^\ell \subseteq H_i$.

{\rm (2)} \ $H$ is a Frobenius algebra with a homogeneous Frobenius homomorphism $\pi : H \lra k$. \end{Lem}

\begin{proof} (1) Let $x = \sum_{j \in \ZZ^n}x_j$ be a non-zero left integral of $H$. Given $h \in H_d$, we have  $\varepsilon(h)x = hx = \sum_{j \in\ZZ^n} hx_j$, so that
\[ hx_j = \varepsilon(h)x_{j+d} \ \ \ \ \forall \ j \in \ZZ^n.\]
Since $\deg \varepsilon = 0$, we obtain $hx_j = 0$ for $d \ne 0$, and $hx_j = \varepsilon(h)x_j$ for $d=0$. Consequently, $x_j \in \int^\ell_H$ for every $j \in \ZZ^n$. Since $\dim_k
\int^\ell_H = 1$, it follows that $\int^\ell_H \subseteq H_i$ for some $i \in \ZZ^n$.

(2) Note that the dual Hopf algebra $H^\ast$ is graded. Let $\pi \in H^\ast$ be a non-zero left integral. A result due to Larson and Sweedler \cite{LS} ensures that $H$ is a Frobenius algebra
with Frobenius homomorphism $\pi : H \lra k$. In view of (1), the linear map $\pi$ is homogeneous. \end{proof}

\bigskip

\begin{Example} Suppose that $\Char(k)=p>0$, and let $(\fg,[p])$ be a finite-dimensional restricted Lie algebra with restricted enveloping algebra $U_0(\fg)$. Assume that $\fg =
\bigoplus_{i\in \ZZ^n}\fg_i$ is restricted graded, that is, $\fg_i^{[p]} \subseteq \fg_{ip}$, so that $U_0(\fg)$ is also $\ZZ^n$-graded. Given a homogeneous basis $\{e_1, \ldots, e_m\}$
of $\fg$, we write $e^a := e_1^{a_1}\cdots e_m^{a_m}$ for every $a \in \NN^m_0$ and put $\tau:=(p\!-\!1,\ldots,p\!-\!1)$ as well as $a\le \tau :\Leftrightarrow a_i \le p\!-\!1$ for
$1\le i\le m$. Then the set $\{e^a \ ; \ 0 \le a \le \tau\}$ is a homogeneous basis of $U_0(\fg)$, and
\[ \pi : U_0(\fg) \lra k \ \ ; \ \ \sum_{0\le a\le \tau}\alpha_ae^a \mapsto \alpha_\tau\]
is a homogeneous Frobenius homomorphism of degree $d_\fg:= -(p\!-\!1)\sum_{i\in \ZZ^n}(\dim_k\fg_i)i$.  Moreover, by \cite[(I.9.7)]{Ja3}, the unique automorphism $\mu :
U_0(\fg)\lra U_0(\fg)$, given by $\mu(x) = x +\tr(\ad x)1$ for all $x \in \fg$, is the Nakayama automorphism corresponding to $\pi$. By Lemma \ref{NF1}, we have
\[ \cN(M) \cong M^{(\mu^{-1})}[d_\fg]\]
for every $M \in \modd \gr U_0(\fg)$. Let $P(S)$ be the projective cover of the graded simple module $S$. General theory implies that
\[ \Soc(P(S)) \cong \cN^{-1}(S) \cong S^{(\mu)}[-d_\fg],\]
which retrieves \cite[(1.9)]{Ja2}. \end{Example}

\bigskip
\noindent
We are going to apply the foregoing result in the context of Frobenius kernels. Suppose that $\Char(k) =p>0$, and let $G$ be a reduced algebraic $k$-group scheme with maximal torus $T$.
We denote the adjoint representation of $G$ on $\Lie(G)$ by $\Ad$. For $r>0$, the algebra $kG_r$ obtains a grading via the adjoint action of $T$ on $kG_r$:
\[ kG_r = \bigoplus_{\lambda \in X(T)} (kG_r)_\lambda.\]
We shall identify characters of $G$, $T$ and $G_r$ with the corresponding elements of the coordinate rings or the duals of the algebras of measures. Let $\lambda_G \in X(G)$ be the character
given by $\lambda_G(g) = \det(\Ad(g))$ for all $g \in G$. If $\fg = \bigoplus_{\alpha \in \Psi\cup \{0\}} \fg_\alpha$ is the root space decomposition of $\fg := \Lie(G)$ relative to $T$,
then
\[ \det(\Ad(t)) = \prod_{\alpha \in \Psi} \alpha(t)^{\dim_k\fg_\alpha}\]
for every $t \in T$, so that
\[ \lambda_G|_T = \sum_{\alpha \in \Psi}(\dim_k\fg_\alpha)\alpha.\]
Given $r>0$, we denote by $\modd G_rT$ the category of finite-dimensional modules of the group scheme $G_rT \subseteq G$. In view of \cite[(2.1)]{Fa5}, which also hols for $r>1$, this
category is a direct sum of blocks of the category $\modd (G_r\!\rtimes\!T)$ of finite-dimensional $(G_r\!\rtimes\!T)$-modules. The latter category is just the category of $X(T)$-graded
$kG_r$-modules. It now follows from work by Gordon-Green \cite[(3.5)]{GG2}, \cite{Gr} that the Frobenius category $\modd G_rT$ affords almost split sequences. By the same token, the
Auslander-Reiten translation $\tau_{G_rT}$ of $\modd G_rT$ is given by
\[ \tau_{G_rT} = \cN\circ \Omega^2_{G_rT}.\]
The following result provides a formula for the Nakayama functor of $\modd G_rT$.

\bigskip

\begin{Prop} \label{HA2} Let $G$ be a reduced algebraic $k$-group scheme with maximal torus $T\subseteq G$.

{\rm (1)} \ The $X(T)$-graded algebra $kG_r$ possesses a Frobenius homomorphism $\pi : kG_r \lra k$ of degree $-(p^r\!-\!1)\lambda_G|_T$.

{\rm (2)} \ We have $ \cN(M) \cong ( M\!\otimes_k\!k_{\lambda_G|_{G_rT}})[-p^r\lambda_G|_T]$ for every $M \in \modd G_rT$. \end{Prop}

\begin{proof} (1) Let $\pi \in k[G_r] = kG_r^\ast$ be a non-zero (left) integral of the commutative Hopf algebra $k[G_r]$. By the proof of \cite[(I.9.7)]{Ja3}, the group $G$ acts on the
subspace $k\pi$ via the character $\lambda_{G_r} : G \lra k \ ; \ g \mapsto \det(\Ad(g))^{-(p^r-1)}$. Consequently, $\pi$ is homogeneous of degree $\lambda_{G_r}|_T  =
-(p^r\!-\!1)\lambda_G|_T\in X(T)$.

(2) Let $\zeta_\ell$ be the left modular function of $kG_r$. By definition, $\zeta_\ell : kG_r \lra k$ is given by $xh = \zeta_\ell(h)x$ for every $h \in kG_r$ and $x \in \int^\ell_{kG_r}$.
Owing to \cite[(1.5)]{FMS}, the automorphism $\id_{kG_r}\ast\zeta_\ell$ is the Nakayama automorphism associated to $(\,,\,)_\pi$. It now follows from \cite[(I.9.7)]{Ja3} that
\[ \zeta_\ell(g) = \lambda_G(g)\]
for all $g \in G_r$. Let $M$ be a graded $kG_r$-module. In view of Lemma \ref{NF1}, the Nakayama functor $\cN$ of $\modd \gr kG_r$ satisfies
\[ \cN(M) \cong  M^{(\id_{kG_r}\ast\zeta_\ell^{-1})}[-(p^r\!-\!1)\lambda_G|_T].\]
For $\lambda \in X(T)$, we denote by $k_\lambda$ the one-dimensional $(G_r\!\rtimes\!T)$-module on which $T$ and $G_r$ act via $\lambda$ and $1$, respectively. Then we have
\[ M[\lambda] \cong M\!\otimes_k\!k_\lambda \ \ \ \ \ \ \forall \ \lambda \in X(T), \, M \in \modd G_r\!\rtimes\!T.\]
For a character $\gamma \in X(G)$, we define characters $\hat{\gamma}, \tilde{\gamma} \in X(G_r\!\rtimes\!T)$ via
\[ \hat{\gamma}(g,t) = \gamma(g)\gamma(t) \ \ \text{and} \ \ \tilde{\gamma}(g,t) = \gamma(g),\]
respectively. Given $M \in \modd G_r\!\rtimes\!T$, we now obtain
\begin{eqnarray*}
 M^{(\id_{kG_r}\ast\zeta_\ell^{-1})}[-(p^r\!-\!1)\lambda_G|_T] & \cong & (M\!\otimes_k\!k_{\tilde{\lambda}_G})\!\otimes_k\!k_{-(p^r-1)\lambda_G|_T}
 \cong (M\!\otimes_k\!k_{\hat{\lambda}_G})\!\otimes_k\!k_{-p^r\lambda_G|_T} \\
 & \cong & (M\!\otimes_k\!k_{\hat{\lambda}_G})[-p^r\lambda_G|_T].
 \end{eqnarray*}
 Let $\omega : G_r\!\rtimes\!T \lra G_rT$ be the canonical quotient map. Since $\hat{\lambda}_G(t^{-1},t) = 1$ for every $t \in T_r$, we have
\[ (\lambda_G|_{G_rT})\circ \omega = \hat{\lambda}_G.\]
As $\modd G_rT$ is a sum of blocks of $\modd G_r\!\rtimes\!T$, it follows that
\[ \cN(M) \cong ( M\!\otimes_k\!k_{\lambda_G|_{G_rT}})[-p^r\lambda_G|_T]\]
for every $M \in \modd G_rT$. \end{proof}

\bigskip
\noindent
For future reference, we record the following result:

\bigskip

\begin{Cor} \label{HA3} Suppose that $G$ is a reduced group scheme with maximal torus $T \subseteq G$ and Lie algebra $\fg = \bigoplus_{\alpha \in \Psi\cup\{0\}}\fg_\alpha$.

{\rm (1)} \ If $\dim_k\fg_\alpha = \dim_k\fg_{-\alpha}$ for every $\alpha \in \Psi$, then $\tau_{G_rT} = \Omega^2_{G_rT}$.

{\rm (2)} \ If $G$ is reductive, then $\tau_{G_rT} = \Omega^2_{G_rT}$.\end{Cor}

\begin{proof} (1) By assumption, we have $\lambda_G|_T \equiv 1$, so that $T \subseteq \ker \lambda_G$. Let $Z_G(T)$ be the Cartan subgroup of $G$ defined by $T$. According to
\cite[(7.2.10)]{Sp}, the group $Z_G(T)$ is connected and nilpotent, and \cite[(6.8)]{Sp} implies that $Z_G(T) \subseteq \ker \lambda_G$. As all Cartan subgroups of $G$ are conjugate,
their union $U$ is also contained in $\ker \lambda_G$. By virtue of \cite[(7.3.3)]{Sp}, the set $U$ lies dense in $G$, so that $\lambda_G \equiv 1$. Our assertion now follows from
Proposition \ref{HA2}.

(2) This is a direct consequence of (1) and \cite[(II.1)]{Ja3}. \end{proof}

\bigskip

\section{Complexity and the Auslander-Reiten Quiver $\Gamma_s(G_rT)$}\label{S:AR}
As before, we fix a smooth group scheme $G$ as well as a maximal torus $T \subseteq G$. The set of roots of $G$ relative to $T$ will be denoted $\Psi$. Recall that the category $\modd
G_rT$ of finite-dimensional $G_rT$-modules is a Frobenius category, whose Heller operator is denoted $\Omega_{G_rT}$. By work of Gordon and Green \cite{GG1}, the forgetful functor $\fF
: \modd G_rT \lra \modd G_r$ preserves and reflects projectives and indecomposables, respectively. Moreover, we have $\Omega^n_{G_r}\circ \fF = \fF \circ \Omega^n_{G_rT}$ for all $n
\in \ZZ$, and the fiber $\fF^{-1}(\fF(M))$, defined by an indecomposable $G_rT$-module $M$, consists of the shifts $\{M\!\otimes_k\!k_{p^r\lambda}\ ; \ \lambda \in X(T)\}$, with
different shifts giving non-isomorphic modules (see \cite{Fa5,FR} for more details). Barring possible ambiguities, we will often suppress the functor $\fF$.

\subsection{Modules of complexity $\bf{1}$} Modules of complexity $1$ play an important r\^ole in the determination of the Auslander-Reiten quiver of a finite group scheme. Our study of
the stable AR-quiver of $\modd G_rT$ also necessitates some knowledge concerning such modules. Since $\modd G_rT$ has projective covers, we have the concept of a minimal projective
resolution, so that we can speak of the complexity $\cx_{G_rT}(M)$ of a $G_rT$-module $M$. Note that $\cx_{G_r}(\fF(M)) = \cx_{G_rT}(M)$ for every $M \in \modd G_rT$.

Recall that the conjugation action of $T$ on $G_r$ induces an operation of $T$ on the variety $V_r(G)$. Standard arguments then show that the rank varieties of $G_rT$-modules are
$T$-invariant subvarieties of $V_r(G)$.

\bigskip

\begin{Thm} \label{MC1} Let $M$ be an indecomposable $G_rT$-module such that $\cx_{G_rT}(M) = 1$. Then there exists a unique $\alpha \in \Psi\cup\{0\}$ such that
\[ \Omega^{2p^{r-\ph(M)}}_{G_rT}(M) \cong M\!\otimes_k\! k_{p^r\alpha}.\] \end{Thm}

\begin{proof} In view of  \cite[(3.2)]{GG1}, the module $M|_{G_r}$ is indecomposable, and we have $\cx_{G_r}(M) = \cx_{G_rT}(M) = 1$. Since $M$ is a $G_rT$-module, the variety
$V_r(G)_M \subseteq V_r(G)$ is $T$-invariant. As a result, the subgroup $\cU_M \subseteq G_r$, provided by Theorem \ref{PH3}, is also $T$-invariant. Thanks to Theorem \ref{PH3}, we
have $s := \ph(M) = \ph_{\cU_M}(M)$.

According to Proposition \ref{PH4}, there exist $\alpha \in \Psi \cup \{0\}$ and $\zeta \in (\HH^{2p^{r-s}}(G_r,k)_{\rm red})_{-p^r\alpha}$ such that $Z(\zeta) \cap \cV_{G_r}(M)
\subsetneq \cV_{G_r}(M)$. Since the module $M|_{G_r}$ is indecomposable, the one-dimensional variety $\cV_{G_r}(M)$ is a line (cf.\ \cite[(7.7)]{SFB2}). Consequently, $Z(\zeta) \cap
\cV_{G_r}(M) = \{0\}$.

In analogy with \S\ref{S:PH}, the cohomology class $\zeta$ can be interpreted as an element of the weight space $\Hom_{G_r}(\Omega^{2p^{r-s}}_{G_rT}(k),k)_{-p^r\alpha}$, or
equivalently, as a non-zero $G_rT$-linear map $\hat{\zeta} : \Omega^{2p^{r-s}}_{G_rT}(k)\!\otimes_k\! k_{-p^r\alpha} \lra k$, see \cite[(I.6.9(5))]{Ja3}. There results an exact sequence
\[ (0) \lra \widehat{L}_\zeta \lra \Omega^{2p^{r-s}}_{G_rT}(k)\!\otimes_k\!k_{-p^r\alpha} \stackrel{\hat{\zeta}}{\lra} k \lra (0)\]
of $G_rT$-modules. Since $\fF(\widehat{L}_\zeta) = L_\zeta$, we have $\cV_{G_r}(\widehat{L}_\zeta\!\otimes_k\!M) = Z(\zeta)\cap \cV_{G_r}(M) = \{0\}$, so that the $G_rT$-module
$\widehat{L}_\zeta\!\otimes_k\!M$ is projective. The arguments of (\ref{PH2}) now yield
\[ \Omega^{2p^{r-s}}_{G_rT}(M)\!\otimes_k\! k_{-p^r\alpha} \cong M,\]
as desired.

If we also have an isomorphism
\[ \Omega^{2p^{r-\ph(M)}}_{G_rT}(M) \cong M\!\otimes_k\! k_{p^r\beta}\]
for some $\beta \in \Psi\cup \{0\}$,  then the shifts $M[p^r\alpha]$ and $M[p^r\beta]$ of the indecomposable $X(T)$-graded module $M$ coincide, and \cite[(4.1)]{GG1} in conjunction
with $X(T)$ being torsion-free yields $\alpha = \beta$. \end{proof}

\bigskip

\begin{Remark} For $r=1$ our result specializes to \cite[(2.4(2))]{Fa5}. \end{Remark}

\bigskip
\noindent
Suppose that $G$ is reductive. For every root $\alpha \in \Psi$, there exists a subgroup $U_\alpha \subseteq G$ on which $T$ acts via $\alpha$. The group $U_\alpha$ is isomorphic to
the additive group $\GG_{a}$ and it is customarily referred to as the {\it root subgroup} of $\alpha$.

An indecomposable $G_rT$-module $M$ is called {\it periodic}, provided there is an isomorphism $\Omega^m_{G_rT}(M)$ $\cong M$ for some $m \in \NN$.

\bigskip

\begin{Cor} \label{MC2} Suppose that $G$ is reductive. Let $M$ be an indecomposable $G_rT$-module. Then the following statements hold:

{\rm (1)} \ If $\cx_{G_rT}(M)=1$, then there exists a unique root $\alpha \in \Psi$ such that $\Omega^{2p^{r-\ph(M)}}_{G_rT}(M) \cong M\!\otimes_k\! k_{p^r\alpha}$.

{\rm (2)} \ The module $M$ is not periodic. \end{Cor}

\begin{proof} (1) Since $M$ is not projective, the Main Theorem of \cite{CPS} provides a root $\alpha \in \Psi$ such that $M|_{(U_\alpha)_r}$ is not projective. Hence we have $\cU_M
= (U_\alpha)_s$ in the proof of Theorem \ref{MC1}.

(2) Assume to the contrary that there exists a natural number $m>0$ such that
\[ \Omega^m_{G_rT}(M) \cong M.\]
Then we have $\cx_{G_rT}(M) = 1$ and part (1) provides a root $\alpha \in \Psi$ and a non-negative integer $\ell \ge 0$ such that
\[ \Omega^{2p^\ell}_{G_rT}(M) \cong M\!\otimes_k\!k_{p^r\alpha}.\]
This implies
\[ M \cong \Omega^{2mp^\ell}_{G_rT}(M) \cong M\!\otimes_k\!k_{mp^r\alpha},\]
which, by our above remarks, is only possible for $m=0$. \end{proof}

\bigskip
\noindent
As noted in Section \ref{S:PH}, the number $2p^{r-\ph(M)}$ is in general only an upper bound for the period of a periodic module. In the classical context of reductive groups, however,  the
periods of restrictions of $G_rT$-modules of complexity $1$ are completely determined by their projective heights.

Let $G$ be reductive. Given a root $\alpha \in \Psi$, we denote by $\alpha^\vee$ the corresponding {\it coroot}, see \cite[(II.1.3)]{Ja3}. As usual, $\rho$ denotes the half-sum of the
positive roots.

\bigskip

\begin{Thm} \label{MC3} Suppose that $G$ is reductive. Let $M \in \modd G_rT$ be indecomposable of complexity $\cx_{G_rT}(M) = 1$. Then $M|_{G_r}$ is periodic with period
${\rm per}(M) = 2p^{r-\ph(M)}$. \end{Thm}

\begin{proof} Let $s$ be the period of the indecomposable module $M|_{G_r}$. By Theorem \ref{MC1}, we have
\[ \Omega_{G_r}^{2p^{r-\ph(M)}}(M) \cong M,\]
so that there exists $\ell �\in \NN$ with $s\ell = 2p^{r-\ph(M)}$. On the other hand, $M$ and $\Omega^s_{G_rT}(M)$ are indecomposable $G_rT$-modules such that
\[ \fF(\Omega^s_{G_rT}(M)) \cong \Omega^s_{G_r}(\fF(M)) \cong \fF(M),\]
and \cite[(4.1)]{GG1} provides $\lambda \in X(T)$ such that
\[ \Omega^s_{G_rT}(M) \cong M\!\otimes_k\!k_\lambda.\]
Since $M$ and $M\!\otimes\!k_\lambda$ both belong to $\modd G_rT$, there exists $\gamma \in X(T)$ such that $\lambda = p^r\gamma$. Applying Theorem \ref{MC1} again, we obtain
\[ M\!\otimes_k\!k_{p^r\alpha} \cong \Omega^{2p^{r-\ph(M)}}_{G_rT}(M) \cong M\!\otimes_k\!k_{\ell p^r\gamma},\]
so that \cite[(4.1)]{GG1} implies $\alpha = \ell \gamma$. Let $W$ be the Weyl group of the root system $\Psi$. General theory (cf.\ \cite[(1.5)]{Hu3}) provides an element $w \in W$ such
that $\alpha_i := w(\alpha)$ is simple. If $\Psi$ is not a union of copies of $A_1$, then there exists a simple root $\alpha_j$ such that $\alpha_i(\alpha_j^\vee) = -1$. Thus, $-1 =
w(\gamma)(\alpha_j^\vee)\ell$, proving that $\ell=1$. Consequently, $s= 2p^{r-\ph(M)}$, as desired.

It remains to consider the case, where the connected components of the root system all have type $A_1$. In that case, we have
\[ 2 = \alpha(\alpha^\vee) = \gamma(\alpha^\vee)\ell,\]
so that $\ell \in \{1,2\}$.

If $\ell = 2$, then $\gamma = \frac{1}{2}\alpha \not \in \ZZ\Psi$. Since the $G_rT$-modules $M$ and $N:=\Omega_{G_rT}^{p^{r-\ph(M)}}(M)$ are indecomposable, they belong to the
same block of $\modd G_rT$, see \cite[(7.1)]{Ja3}. Consequently, all their composition factors belong to the same linkage class. Let $W_p$ be the affine Weyl group and denote by
\[ w\dact \lambda := w(\lambda\!+\!\rho)-\rho \ \ \ \ \forall \ w \in W_p,\, \lambda \in X(T)\]
the dot action of $W_p$ on $X(T)$, cf.\ \cite[(II.6.1)]{Ja3}. Suppose that $\widehat{L}_r(\mu)$ is a composition factor of $M$. According to \cite[(II.9.6)]{Ja3}, the module
$\widehat{L}_r(\mu\!+\!p^r\gamma) \cong \widehat{L}_r(\mu)\!\otimes\!k_{p^r\gamma}$ is a composition factor of $N \cong M\!\otimes_k\!k_{p^r\gamma}$. The linkage principle
\cite[(II.9.19)]{Ja3} implies that $\mu+p^r\gamma \in W_p\dact\mu$. Since $w\dact\lambda \equiv \lambda \ \ \modd(\ZZ\Psi)$ for all $w \in W_p$ and $\lambda \in X(T)$, we
conclude that $p^r\gamma = \frac{p^r}{2}\alpha \in \ZZ\Psi$. As $p \ge 3$, we have reached a contradiction.

As a result, $\ell=1$, so that we have $s=2p^{r-\ph(M)}$ in this case as well. \end{proof}

\bigskip
\noindent
We turn to the question, which periods can actually occur. Our approach necessitates the following realizability criterion.

\bigskip

\begin{Prop} \label{MC4} Suppose that $V \subseteq V_r(G)$ is a $T$-invariant conical closed subvariety. Then there exists a $G_rT$-module $M$ such that $V = V_r(G)_M$. \end{Prop}

\begin{proof} Thanks to \cite[(6.8)]{SFB2}, it suffices to establish the corresponding result for support varieties. Since $T_r$ acts trivially on $\HH^\bullet(G_r,k)$ (cf.\ \cite[(I.6.7)]{Ja3}), the $T$-action on $\HH^\bullet(G_r,k)$ gives rise to the following decomposition
\[ \HH^\bullet(G_r,k) = \bigoplus_{n\ge 0}\HH^{2n}(G_r,k) = \bigoplus_{n\ge 0}\bigoplus_{\lambda \in X(T)} \HH^{2n}(G_r,k)_{p^r\lambda}.\]
If $V\subseteq \cV_{G_r}(k)$ is a conical, $T$-invariant closed subvariety, then there exists a homogeneous $T$-invariant ideal $I_V \unlhd \HH^\bullet(G_r,k)$ such that
\[ V = Z(I_V).\]
We let $\zeta_1, \ldots, \zeta_r$ be homogeneous generators of $I_V$, that is, $\zeta_i \in \HH^{2n_i}(G_r,k)_{p^r\lambda_i}$ for suitable $n_i\ge 0$ and $\lambda_i \in X(T)$. As noted
earlier, each $\zeta_i$ corresponds to a map $\hat{\zeta}_i : \Omega^{2n_i}_{G_rT}(k)\!\otimes_k\!k_{-p^r\lambda_i} \lra k$, whose kernel $\widehat{L}_{\zeta_i} := \ker \hat{\zeta}_i
\in \modd G_rT$ has support
\[ \cV_{G_r}(\widehat{L}_{\zeta_i}) = Z(\zeta_i),\]
see \cite[(7.5)]{SFB2}. The result now follows from the tensor product theorem \cite[(7.2)]{SFB2}. \end{proof}

\bigskip

\begin{Cor} \label{MC5} Let $G$ be reductive. Then the following statements hold:

{\rm (1)} \ Given $s \in \{0,\ldots,r\!-\!1\}$, there exists an indecomposable, periodic $G_r$-module $M$, whose period equals $2p^s$.

{\rm (2)} \ The stable AR-components $\Theta \subseteq \Gamma_s(G_r)$ containing a module of complexity $1$ are precisely of the form $\ZZ[A_\infty]/\langle \tau^{p^s}\rangle$,
where $s \in \{0,\ldots,r\!-\!1\}$.\end{Cor}

\begin{proof} (1) Let $\alpha \in \Psi$ be a root, $U_\alpha \subseteq G$ be the corresponding root subgroup. We consider the subgroup
\[ \cU := (U_\alpha)_{r-s} \subseteq G_r\]
and note that $\cU \cong \GG_{a(r-s)}$ is a $T$-invariant elementary abelian subgroup of height $\height(\cU) = r\!-\!s$. Let $\varphi : \GG_{a(r)} \lra \cU$ be the map such that $\im
\varphi = \cU$. Since $T$ acts on $\cU$ via the character $\alpha$, the one-dimensional closed subvariety
\[ V := \{ \gamma\dact \varphi \ ; \ \gamma \in k^\times\} \cup \{0\}\]
of $V_r(G)$ is $T$-invariant. Proposition \ref{MC4} now provides $N \in \modd G_rT$ such that $V_r(G)_N = V$. Since $V$ is irreducible, we have $V_r(G)_M =V$ for a suitable indecomposable constituent $M$ of $N$. According to Theorem \ref{MC3}, the module $M|_{G_r}$ is periodic, with period $2p^{r-\height(\cU)}=2p^s$.

(2) Since $kG_r$ is symmetric, we have $\tau_{G_r} \cong \Omega^2_{G_r}$. A consecutive application of \cite[(2.1)]{Fa3}, \cite[(4.1)]{Fa3} and \cite[(4.4)]{Fa5} shows that $\Theta$ is
of the form $\ZZ[A_\infty]/\langle \tau^{p^s}\rangle$. Part (1) implies that for every $s \in\{0,\ldots,r\!-\!1\}$ there exists an infinite tube of rank $p^s$. \end{proof}

\bigskip

\begin{Remark} The example of the groups $\SL(2)_1T_r$ with $r\ge 3$ shows that for infinitesimal groups that are not Frobenius kernels of smooth groups, the ranks of tubes may be more
restricted, see \cite[(5.6)]{FV2}. \end{Remark}

\bigskip

\subsection{Webb's Theorem} Results by Happel-Preiser-Ringel \cite{HPR} show that the presence of so-called {\it subadditive functions} imposes constraints on the structure of the tree class of a connected stable representation quiver. This approach was first effectively employed by Webb \cite{We} in his determination of the tree classes for AR-components of group algebras of
finite groups. We shall establish an analogue for $\modd G_rT$, with a refinement for the case, where $G$ is reductive.

Let $G$ be a smooth algebraic group scheme. In the sequel, we let $\Gamma_s(G_rT)$ be the stable Auslander-Reiten quiver of the Frobenius category $\modd G_rT$. For a component $\Theta
\subseteq \Gamma_s(G_rT)$, we have
\[ V_r(G)_{\fF(M)} = V_r(G)_{\fF(N)} \ \ \ \ \ \text{for all} \ M,N \in \Theta.\]
Accordingly, we can attach a variety $V_r(G)_\Theta$ to the component $\Theta$. By combining this fact with results by Happel-Preiser-Ringel \cite{HPR} one obtains:

\bigskip

\begin{Prop}[cf.\ \cite{Fa5}] \label{WT1} Let $\Theta \subseteq \Gamma_s(G_rT)$ be a component. Then the tree class $\bar{T}_\Theta$ is a simply laced finite or infinite Dynkin diagram,
a simply laced Euclidean diagram, or $\tilde{A}_{12}$. \hfill $\square$ \end{Prop}

\bigskip

\begin{Prop} \label{WT2} Suppose that $G$ is reductive, and let $\Theta \subseteq \Gamma_s(G_rT)$ be a component such that $\dim V_r(G)_\Theta \ne 2$. Then $\Theta \cong
\ZZ[A_\infty]$. \end{Prop}

\begin{proof} A consecutive application of Corollary \ref{HA3} and Corollary \ref{MC2} shows that $\modd G_rT$ has no $\tau_{G_rT}$-periodic modules. We may thus adopt the arguments
of \cite[(3.2)]{Fa5}. \end{proof}

\bigskip
\noindent
In case the underlying group is reductive, only three types of components can occur:

\bigskip

\begin{Thm} \label{WT3} Let $G$ be reductive of characteristic $\Char(k)=p\ge 3$. Suppose that $\Theta \subseteq \Gamma_s(G_rT)$ is a component.

{\rm (1)} \ If $\Theta$ contains a simple module $S$ of complexity $\cx_{G_rT}(S)=2$, then $\Theta \cong  \ZZ[A_\infty], \ZZ[A_\infty^\infty]$.

{\rm (2)} \ We have $\Theta \cong \ZZ[A_\infty]$, $\ZZ[A_\infty^\infty]$, or $\ZZ[D_\infty]$. \end{Thm}

\begin{proof} (1) In view of  \cite[(1.3)]{GG1}, the module $S|_{G_r}$ is simple and of complexity $\cx_{G_r}(S) = 2$. In this situation \cite[(7.1)]{Fa2} provides a decomposition $G =
HK$ of $G$ into an almost direct product such that

(a) \ $\fg = \Lie(H) \oplus \Lie(K)$ with $\Lie(H) = \fsl(2)$, and

(b) \ $V_r(G)_M = V_r(H)_M$ and $M|_{K_r}$ is projective for every $M \in \Theta$.

\noindent
Since $H$ is an almost simple group of rank $1$, it follows that the central subgroup $H\cap K \subseteq H$ is either trivial or isomorphic to $\mu_{(2)}$ (cf.\ \cite[(8.2.4)]{Sp}). Hence, if
$H\cap K \ne e_k$, then  there exists a character $\lambda \in X(H\cap K) \cong \ZZ/(2)$ such that $H\cap K$ acts on every vertex $M \in \Theta$ via $\lambda$. As $H\cap K$ is contained
in the maximal torus $T$, we can find $\gamma \in X(T)$ with $\gamma|_{H\cap K} = \lambda$. In view of $p$ being odd, we also have $p^r\gamma|_{H\cap K} = \lambda$.
Consequently, $H\cap K$ acts trivially on every vertex of the shifted component $\Theta[-p^r\gamma] \cong \Theta$.

Let $G' := G/(H\cap K)$, and consider its maximal torus $T' := T/(H\cap K)$ (cf.\ \cite[(7.2.7)]{Sp}). According to \cite[(II.9.7)]{Ja3}, there results an exact sequence
\[ e_k \lra H\cap K \lra G_rT \lra G'_rT' \lra e_k,\]
with $\modd G'_rT'$ being a sum of blocks of $\modd G_rT$. By the above observation, a suitable shift of $\Theta$ belongs to $\modd (G'_rT')$. Setting $H' := H/(H\cap K)$ and
$K' := K/(H\cap K)$, we have $G' = H' \times K'$, while (a) and (b) continue to hold for $H'$ and $K'$. As a result, we may assume in addition that

(c) \ $G = H \times K$.

\noindent
By general theory, there exist maximal tori $T_H \subseteq H$ and $T_K \subseteq K$ such that $T = T_H\times T_K$.

Consequently, the isomorphism
\[kG_r \cong kH_r\!\otimes_k\!kK_r\]
induced by (c) is compatible with the $T$-action. Thus, the outer tensor product defines a functor
\[ \modd H_rT_H \times \modd K_rT_K \lra \modd G_rT  \ \ ; \ \ (M,N) \mapsto M\!\otimes_k\!N.\]

Let $\cB \subseteq kG_r$ be the block containing the simple $G_r$-module $S$, so that $\Theta \subseteq \Gamma_s(\cB T)$. Owing to \cite[Section 10.E]{CR}, the $G_r$-module $S$ is an
outer tensor product
\[ S \cong S_1\!\otimes_k\!S_2\]
with a simple projective $K_r$-module $S_2$. In view of \cite[(II.9.6)]{Ja3}, we may assume that $S_1 \in \modd H_rT$ and $S_2\in \modd K_rT$. It now follows from \cite[(4.1)]{GG1}
that
\[S[\gamma] \cong S_1\!\otimes_k\!S_2\]
for a suitable $\gamma \in X(T)$. (Since the right-hand module lies in $\modd G_rT$, we actually have $\gamma \in p^rX(T)$.)

Letting $\cB_1 \subseteq kH_r$ be the block containing $S_1$, we obtain inverse equivalences
\[ \modd \cB_1 \lra \modd \cB \ \ ; \ \ X \mapsto X\!\otimes_k\!S_2 \]
and
\[ \modd \cB \lra \modd \cB_1 \ \ ; \ \ Y \mapsto \Hom_{kK_r}(S_2,Y),\]
so that the first functor induces an equivalence
\[  \modd \cB_1T_H \lra \modd \cB T \ \ ; \ \ X \mapsto X\!\otimes_k\!S_2. \]
Thus, $\Theta$ is isomorphic to a component $\Theta_1 \subseteq \Gamma_s(H_rT_H)$, which, by (b), has a two-dimensional rank variety.

Since $H$ has rank $1$, we have $H \cong \SL(2), \PSL(2)$. Since $\modd \PSL(2)_{r-s}T'$ is a sum of blocks of $\modd \SL(2)_{r-s}T$, it suffices to address the case, where $H = \SL(2)$.
We shall write $T := T_H$. As noted in Section $5$, the block $\cB_1$ is of the form $\cB^{(r)}_{i,s}$ and Lemma \ref{Rep1} provides a Morita equivalence
\[ \modd \cB_{i,0}^{(r-s)} \lra \modd \cB^{(r)}_{i,s} \ \ ; \ \ M \mapsto \St_s\!\otimes_k\!M^{[s]}.\]
Thanks to \cite[(II.10.4)]{Ja3}, this functor and its inverse take $\SL(2)_{r-s}T$-modules to $\SL(2)_rT$-modules, so that $\Theta_1$ is isomorphic to a component $\Theta_2$ of
$\Gamma_s(\SL(2)_{r-s}T)$, whose modules have complexity $2$.

If $r\!-\!s=1$, then standard $\SL(2)_1$-theory (see \cite[\S3, Example]{Fa5}) implies $\Theta_2 \cong \ZZ[A_\infty^\infty]$.   Now assume that $r\!-\!s \ge 2$ and let
$\widehat{L}_{r-s}(\lambda)$ be the simple module belonging to $\Theta_2$ which corresponds to $S$. According to \cite[(3.3)]{FR}, the forgetful functor $\fF : \modd \SL(2)_{r-s}T \lra
\modd \SL(2)_{r-s}$ takes our component $\Theta_2$ to the component $\Psi_2 := \fF(\Theta_2) \subseteq \Gamma_s(\SL(2)_{r-s})$ containing the simple module $L_{r-s}(\lambda)$.
Since $r-s\ge 2$, Lemma \ref{Rep3} shows that $\Ht_{r-s}(\lambda)$ is indecomposable. From the standard almost split sequence
\[ (0) \lra \Rad(P_{r-s}(\lambda)) \lra \Ht_{r-s}(\lambda) \oplus P_{r-s}(\lambda)\lra P_{r-s}(\lambda)/\Soc(P_{r-s}(\lambda)) \lra (0)\]
we see that $ P_{r-s}(\lambda)/\Soc(P_{r-s}(\lambda))$ has exactly one predecessor. Consequently, the module $L_{r-s}(\lambda)$ $ \cong \Omega_{\SL(2)_{r-s}}(
P_{r-s}(\lambda)/\Soc(P_{r-s}(\lambda)))$ enjoys the same property and Proposition \ref{Rep5} guarantees that $\Psi_2 \cong \ZZ[A_\infty]$. It follows that $\Theta_2 \cong
\ZZ[A_\infty], \ZZ[A_\infty^\infty], \ZZ[D_\infty]$. Let $\varphi \in V_{r-s}(G)_{\Theta_2}$ and consider the module $M_\varphi := k\SL(2)_{r-s}\!\otimes_{k[u_{r-1}]}k$. As argued in
\cite[(3.2)]{Fa5},
\[ \delta : \Psi_2 \lra \NN \ \ ; \ \ X \mapsto \dim_k\Ext^1_{\SL(2)_{r-s}}(M_\varphi,X)\]
is a $\tau_{\SL(2)_{r-s}}$-invariant subadditive function such that $\delta \circ \fF$ is a $\tau_{\SL(2)_{r-s}T}$-invariant subadditive function on $\Theta_2$. If $\bar{T}_{\Theta_2} \in \{
A_\infty^\infty, D_\infty\}$, then \cite[(VII.3.4)]{ARS} and \cite[(VII.3.5)]{ARS} show
that $\delta \circ \fF$, and thereby $\delta$, is bounded and additive. Since $\ZZ[A_\infty]$ does not afford such a function, it follows that $\Theta_2 \cong \ZZ[A_\infty]$, as desired.

(2) In view of Proposition \ref{WT2}, we may assume that $\dim V_r(G)_\Theta = 2$. By Proposition \ref{WT1}, it remains to rule out the case, where $\Theta \cong
\ZZ[\tilde{A}_{pq}]$ or where $\bar{T}_\Theta$ is Euclidean. In these cases, the component $\Theta$ has only finitely many $\tau_{G_rT}$-orbits, so that the component $\Psi :=
\fF(\Theta)$ also enjoys this property. It now follows from \cite[(4.1)]{Fa3} that $\Psi \cong \ZZ[\tilde{A}_{12}]$. Thanks to \cite[Thm.A]{We}, we may assume that $\Theta$ contains a
simple module, and (1) shows that the abovementioned cases cannot occur. \end{proof}

\bigskip

\begin{Remark} It is not known whether components of tree class $D_\infty$ actually occur. According to \cite[(4.5)]{FR}, components containing baby Verma modules have tree class
$A_\infty$. \end{Remark}

\bigskip

\subsection{Components containing Verma modules} Throughout, $G$ denotes a smooth reductive group scheme with maximal torus $T \subseteq G$ and root system $\Psi$. By picking
a Borel subgroup $B \subseteq G$ containing $T$ we obtain the sets $\Psi^+$ and $\Sigma$ of positive and simple roots, respectively. Given $\lambda \in X(T)$, we denote by $k_\lambda$
the corresponding one-dimensional $T$-module. Since $B = UT$ is a product of $T$ and the unipotent radical $U \subseteq B$, this module is also a $B$-module. Given $r\in \NN$, the adjoint
representation endows the induced $\Dist(G_r)$-module
\[ Z_r(\lambda) := \Dist(G_r)\!\otimes_{\Dist(B_r)}\!k_\lambda\]
with the structure of a $G_rT$-module. We denote this module by $\widehat{Z}_r(\lambda)$ and refer to $Z_r(\lambda)$ and $\widehat{Z}_r(\lambda)$ as {\it baby Verma modules}
defined by $\lambda$. Given $\lambda \in X(T)$, the modules $Z_r(\lambda)$ and $\widehat{Z}_r(\lambda)$ have simple tops $L_r(\lambda)$ and $\widehat{L}_r(\lambda)$, and all
simple objects of $\modd G_r$ and $\modd G_rT$ arise in this fashion. Moreover, every simple $G_rT$-module $\widehat{L}_r(\lambda)$ has a projective cover $\widehat{P}_r(\lambda)$,
see \cite[\S II.3, \S II.9, \S II.11]{Ja3} for more details. In what follows, $B^-$ denotes the Borel subgroup opposite to $B$.

The main result of this section, Theorem \ref{VM3}, employs support varieties to study the Heller translates of the baby Verma modules $\widehat{Z}_r(\lambda)$. This is motivated by the
Auslander-Reiten theory of the Frobenius categories $\modd G_rT$ and $\modd G_r$, where non-projective Verma modules are quasi-simple (cf.\ \cite[(4.4),(4.5)]{FR}). Our result implies
that the connected components of the stable Auslander-Reiten quiver $\Gamma_s(G_r)$ contain at most one baby Verma module $Z_r(\lambda)$.

We let $\cF(\Delta) \subseteq \modd G_rT$ be the subcategory of {\it $\Delta$-good} modules. By definition, every object $M \in \cF(\Delta)$ possesses a filtration, a so-called {\it $\widehat{Z}_r$-filtration}, whose factors are baby Verma modules $\widehat{Z}_r(\lambda)$. The filtration multiplicities $[M\!:\!\widehat{Z}_r(\lambda)]$ do not depend on the
choice of the filtration, and each projective indecomposable module $\widehat{P}_r(\lambda)$ belongs to $\cF(\Delta)$, with its filtration multiplicities being linked to the Jordan-H\"older multiplicities by BGG reciprocity:
\[ [\widehat{P}_r(\lambda)\!:\!\widehat{Z}_r(\mu)] = [\widehat{Z}_r(\mu)\!:\!\widehat{L}_r(\lambda)],\]
see \cite[\S II.11]{Ja3}. Since $[\widehat{Z}_r(\lambda)\!:\!\widehat{L}_r(\lambda)] = 1$, this readily implies the following:

\bigskip

\begin{Lem} \label{VM1} Let $m>0$. Then $\Omega^m_{G_rT}(\widehat{Z}_r(\lambda)) \in \cF(\Delta)$ with filtration factors $\widehat{Z}_r(\mu)$ for $\mu>\lambda$.
\hfill $\square$ \end{Lem}

\bigskip
\noindent
We require the following subsidiary result concerning the subset $\wt(M) \subseteq X(T)$ of weights of a $G_rT$-module $M$.

\bigskip

\begin{Lem} \label{VM2} Let $M$ be a $G_rT$-module such that $\wt(M) \subseteq \lambda + \ZZ\Psi$ for some $\lambda \in X(T)$. Then $\wt(\Omega_{G_rT}(M)) \subseteq
\lambda + \ZZ\Psi$. \end{Lem}

\begin{proof} We let $T$ act on $\Dist(G_r)$ via the adjoint representation and put $A := \Dist(U_r^-)\Dist(U_r)$. Then $A$ is a $T$-submodule of $\Dist(G_r)$ with $\wt(A) \subseteq \ZZ\Psi$ (cf.\ \cite[(II.1.19)]{Ja3}).

Given $\gamma \in X(T)$, we consider the $G_r$-module
\[P_\gamma := \Dist(G_r)\!\otimes_{\Dist(T_r)}\!k_\gamma.\]
In view of \cite[(II.1.12)]{Ja3}, the adjoint action of $T$ endows $P_\gamma$ with the structure of a $G_rT$-module such that
\[ P_\gamma|_T \cong A\!\otimes_k\!k_\gamma.\]
Frobenius reciprocity yields $\Ext^1_{G_r}(P_\gamma,-)\cong \Ext^1_{T_r}(k_\gamma,-) = 0$, so that each $P_\gamma$ is a projective $G_rT$-module with $\wt(P_\gamma) \subseteq
\gamma + \ZZ\Psi$, see \cite[(II.9.4)]{Ja3}. The canonical surjection $P_\gamma \lra \widehat{Z}_r(\gamma)$ induces a surjective map $P_\gamma \lra \widehat{L}_r(\gamma)$.
Consequently, the projective cover $\widehat{P}_r(\gamma)$ of $\widehat{L}_r(\gamma)$ has weights $\wt(\widehat{P}_r(\gamma)) \subseteq \gamma + \ZZ\Psi$.

Let $\widehat{P}(M)$ be the projective cover of the $G_rT$-module $M$. By the above, we have $\wt(\widehat{P}(M)) \subseteq \lambda +\ZZ\Psi$, whence
\[ \wt(\Omega_{G_rT}(M)) \subseteq \wt(\widehat{P}(M)) \subseteq \lambda +\ZZ\Psi,\]
as desired. \end{proof}

\bigskip

\begin{Thm} \label{VM3} Suppose that $G$ is defined over $\FF_p$ and that $p$ is good for $G$. Let $\lambda, \mu \in X(T)$ be characters such that there exists $m >0$ with
\[ \Omega^{2m}_{G_rT}(\widehat{Z}_r(\lambda)) \cong \widehat{Z}_r(\mu).\]
Then the following statements hold:

{\rm (1)} \ We have $\dep(\lambda)=\dep(\mu) = r$, and there exists a simple root $\alpha \in \Sigma\setminus \Psi^r_\lambda$ such that $\mu = \lambda + mp^r\alpha$.

{\rm (2)} \ $\Omega^{2n}_{G_rT}(\widehat{Z}_r(\lambda)) \cong \widehat{Z}_r(\lambda\!+\!np^r\alpha)$ for all $n \in \ZZ$. \end{Thm}

\begin{proof} (1) We first assume that $\dep(\lambda) = 1$. According to \cite[(5.2)]{FR}, there exists a simple root $\alpha \in \Sigma\setminus \Psi_\lambda^1$ with
$\cV_{(U_\alpha)_r}(k) \subseteq \cV_{G_r}(\widehat{Z}_r(\lambda))$.

Proposition \ref{PH4}(2) provides an element $\zeta \in (\HH^2(G_r,k)_{\rm red})_{-p^r\alpha}$ such that $Z(\zeta)\cap \cV_{G_r}(\widehat{Z}_r(\lambda)) \subsetneq
\cV_{G_r}(\widehat{Z}_r(\lambda))$. Then $\eta := \zeta^m \in (\HH^{2m}(G_r,k)_{\rm red})_{-mp^r\alpha}$ also has this property, and there results a short exact sequence
\[ (0) \lra \widehat{L}_\eta \lra \Omega^{2m}_{G_rT}(k)\!\otimes_k\!k_{-mp^r\alpha} \stackrel{\hat{\eta}}{\lra} k \lra (0).\]
By tensoring this sequence with $\widehat{Z}_r(\lambda)$ while observing \cite[(II.9.2)]{Ja3}, we obtain a short exact sequence
\[ (\ast) \ \ \ \ \ \ \ \  (0) \lra \widehat{L}_\eta\! \otimes_k\!\widehat{Z}_r(\lambda) \stackrel{\binom{g_1}{g_2}}{\lra} \widehat{Z}_r(\mu\!-\!mp^r\alpha)\oplus
(\text{proj.})\stackrel{(f_1,f_2)}{\lra}  \widehat{Z}_r(\lambda) \lra (0)\]
of $G_rT$-modules. In particular, we have
\[ \widehat{Z}_r(\lambda) = f_1(\widehat{Z}_r(\mu\!-\!mp^r\alpha)) + f_2((\text{proj.})).\]
Since the baby Verma module $\widehat{Z}_r(\lambda)$ has a simple top, at least one summand has to coincide with $\widehat{Z}_r(\lambda)$.

\medskip
(a) \ {\it If $f_2(({\rm proj.})) \ne \widehat{Z}_r(\lambda)$, then $r=1$ and $\mu =\lambda+mp\alpha$}.

\smallskip
\noindent
In this case, we have $f_1(\widehat{Z}_r(\mu\!-\!mp^r\alpha)) = \widehat{Z}_r(\lambda)$, so that equality of dimensions implies
\[\widehat{Z}_r(\mu\!-\!mp^r\alpha) \cong \widehat{Z}_r(\lambda),\]
whence $\mu = \lambda+mp^r\alpha$. In view of \cite[(II.3.7(9))]{Ja3}, we thus have $Z_r(\mu) \cong Z_r(\lambda)$, while our assumption implies
\[ Z_r(\lambda) \cong Z_r(\mu) \cong \Omega^{2m}_{G_r}(Z_r(\lambda)).\]
Consequently, $\cx_{G_r}(Z_r(\lambda))=1$ and the inequality
\[ r = \dim \cV_{(U_\alpha)_r}(k) \le \dim \cV_{G_r}(Z_r(\lambda)) = 1\]
gives $r=1$ and $\mu = \lambda\!+\!mp\alpha$. \ \ \ \ $\Diamond$

\medskip
\noindent
{\it In view of (a) we henceforth assume that $f_2(({\rm proj.})) = \widehat{Z}_r(\lambda)$}.

\medskip
(b) \ {\it We have $\lambda-2(p^r-1)\rho \le \mu-mp^r\alpha \le \lambda$}.

\smallskip
\noindent
If $f_1 = 0$, then our exact sequence ($\ast$) yields
\[ \widehat{L}_\eta\!\otimes_k\!\widehat{Z}_r(\lambda) \cong \ker(0,f_2) \cong \widehat{Z}_r(\mu\!-\!mp^r\alpha)\oplus \Omega_{G_rT}(\widehat{Z}_r(\lambda))
\oplus (\text{proj.}),\]
whence
\begin{eqnarray*}
Z(\eta)\cap \cV_{G_r}(\widehat{Z}_r(\lambda)) & = & \cV_{G_r}(\widehat{Z}_r(\mu\!-\!mp^r\alpha)) \cup \cV_{G_r}(\Omega_{G_rT}(\widehat{Z}_r(\lambda)))\\
& = & \cV_{G_r}(\widehat{Z}_r(\mu)) \cup \cV_{G_r}(\Omega_{G_rT}(\widehat{Z}_r(\lambda)))\\
& = & \cV_{G_r}(\widehat{Z}_r(\lambda)),
\end{eqnarray*}
a contradiction. Thus, $f_1 \ne 0$ and $\widehat{Z}_r(\lambda)_{\mu-mp^r\alpha} \ne (0)$, so that \cite[(II.9.2(6))]{Ja3} implies $\lambda-2(p^r-1)\rho \le \mu-mp^r\alpha \le
\lambda$. \ \ \ \ $\Diamond$

\medskip
(c) \ {\it There exists $0<n \le mp^r$ such that $\mu = \lambda+n\alpha$}.

\smallskip
\noindent
Since $m>0$, it readily follows from Lemma \ref{VM1} that $\mu>\lambda$. In view of (b), we therefore have
\[ \mu-mp^r\alpha \le \lambda < \mu\]
so that there exist non-negative integers $n_\beta$, $m_\beta$ with
\[ \mu = \lambda + \sum_{\beta \in \Sigma} n_\beta\beta \ \ \text{and} \ \ \lambda = \mu-mp^r\alpha + \sum_{\beta \in \Sigma}m_\beta \beta.\]
Consequently, $n_\alpha+m_\alpha = mp^r$ while $n_\beta + m_\beta = 0$ for every simple root $\beta \ne \alpha$. \ \ \  \ $\Diamond$

\medskip
\noindent
We let $L$ be the Levi subgroup of $G$ that is defined by the simple root $\alpha$. The baby Verma module of $L_rT$, associated to the weight $\lambda \in X(T)$ will be denoted
$\widehat{Z}_r^L(\lambda)$.

\medskip
(d) \ {\it We have $\Omega^{2m}_{L_rT}(\widehat{Z}_r^L(\lambda)) \cong \widehat{Z}_r^L(\mu)$}.

\smallskip
\noindent
Assuming $L\ne G$, we consider the triangular decomposition $G_r = N^{-}_rL_rN_r$ of $G_r$, see \cite[(II.3.2)]{Ja3}. For $\gamma \in X(T)$, we have an isomorphism
\[ \widehat{Z}_r( \gamma)|_{L_rT} \cong \Dist(N_r^-)_{\rm ad}\!\otimes _k\!\widehat{Z}_r^L(\gamma) \cong \widehat{Z}_r^L(\gamma)\oplus W(\gamma),\]
with $W(\gamma) := \Dist(N_r^-)_{\rm ad}^\dagger\!\otimes _k\!\widehat{Z}_r^L(\gamma)$ being defined via the augmentation ideal of $\Dist(N^-_r)$. The subscript indicates that
$L_rT$ acts via the adjoint representation (cf.\ \cite[(II.3.6(2))]{Ja3}). Consequently, (c) yields
\[ \wt(W(\mu)) \subseteq \bigcup_{\gamma \in X(T)\setminus (\mu +\ZZ\alpha)} \gamma + \ZZ\alpha = \bigcup_{\gamma \in X(T)\setminus (\lambda+\ZZ\alpha)} \gamma +
\ZZ\alpha.\]
General properties of the Heller operator give rise to
\[ \widehat{Z}_r^L(\mu) \oplus W(\mu) \cong \widehat{Z}_r(\mu)|_{L_rT} \cong \Omega^{2m}_{L_rT}(\widehat{Z}_r^L(\lambda))\oplus \Omega^{2m}_{L_rT}(W(\lambda))
\oplus (\text{proj.}).\]
According to \cite[(4.2.1)]{NPV}, we obtain
\[ \cV_{L_r}(\widehat{Z}_r^L(\lambda)) =  \cV_{L_r}(\widehat{Z}_r(\lambda)) = \cV_{G_r}(\widehat{Z}_r(\lambda))\cap \cV_{L_r}(k) \supseteq \cV_{(U_\alpha)_r}(k),\]
so that $\widehat{Z}_r^L(\lambda)$ is not projective and $\Omega^{2m}_{L_rT}(\widehat{Z}_r^L(\lambda))$ is indecomposable. Since $\wt(\widehat{Z}_r^L(\lambda)) \subseteq
\lambda\! +\!\ZZ\alpha$, Lemma \ref{VM2} ensures that $\wt(\Omega^{2m}_{L_rT}(\widehat{Z}_r^L(\lambda))) \subseteq \lambda + \ZZ\alpha$. As a result, the indecomposable
$L_rT$-module $\Omega^{2m}_{L_rT}(\widehat{Z}_r^L(\lambda))$ is not a direct summand of $W(\mu)$.  The Theorem of Krull-Remak-Schmidt thus yields
$\Omega^{2m}_{L_rT}(\widehat{Z}_r^L(\lambda)) \cong \widehat{Z}_r^L(\mu)$. \ \ \ \ $\Diamond$

\medskip
(e) \ {\it Let $\gamma \in X(T)$. If $\widehat{Z}_r^L(\gamma)$ is not projective, then $\widehat{Z}_r^L(\gamma)|_{L_1T}$ has no non-zero projective summands.}

\smallskip
\noindent
The semi-simple part of $L$ has rank $1$ and is therefore isomorphic to $\SL(2)$ or ${\rm PSL}(2)$, see \cite[(8.2.4)]{Sp}. (Since parabolic subgroups are connected (cf.\
\cite[(7.3.8)]{Sp}), so are Levi subgroups.) In view of \cite[(II.9.7)]{Ja3}, the category $\modd {\rm PSL}(2)_rT$ is a sum of blocks of $\modd \SL(2)_rT$ (see also \cite[(3.5)]{FR}). We may therefore assume without loss of generality that $L_rT \cong \SL(2)_rT$.

Let $\widehat{P}$ be a projective indecomposable $L_1T$-module, which is a direct summand of $\widehat{Z}_r^L(\gamma)|_{L_1T}$. If $\widehat{P}$ is not simple, then standard
$\SL(2)_1T$-theory (cf.\ \cite[(12.2)]{Hu4}) shows that $\widehat{P}$ possesses weights of multiplicity $\ge 2$. Since every weight of $\widehat{Z}_r^L(\gamma)|_{L_1T}$ has
multiplicity $1$, we conclude that $\widehat{P}$ is simple, and hence is of the form $\widehat{Z}^L_1(\omega)$ with $\langle \omega+\rho,\alpha^\vee\rangle \in p\ZZ$, see
\cite[(II.11.8)]{Ja3}. Writing $\gamma = \gamma_0 + p\gamma_1$ with $\gamma_0 \in X_1(T)$, it follows from \cite[(5.4)]{FR} that $\widehat{Z}_r^L(\gamma)|_{L_1}$ belongs to the block $\cB_1(\omega) \subseteq \Dist(L_1)$, defined by $\omega$. Since $\cB_1(\omega)$ is simple, another application of \cite[(5.4)]{FR} implies the projectivity of $\widehat{Z}_r^L(\gamma)|_{L_1}$, which, by \cite[(II.9.4)]{Ja3} and \cite[(II.11.8)]{Ja3}, yields a contradiction. \ \ \ \ $\Diamond$

\medskip
(f) \ {\it We have $\mu = \lambda +pm\alpha$}.

\smallskip
\noindent
Owing to (d), the module $\widehat{Z}_r^L(\mu)$ is not projective. Standard properties of the Heller operator in conjunction with (e) thus yield
\[ \widehat{Z}_r^L(\mu)|_{L_1T} \cong \Omega^{2m}_{L_rT}(\widehat{Z}_r^L(\lambda))|_{L_1T} \cong \Omega^{2m}_{L_1T}(\widehat{Z}_r^L(\lambda)|_{L_1T}).\]
By the same token, every indecomposable summand of $\widehat{Z}_r^L(\lambda)|_{L_1T}$ is non-projective and hence of complexity $1$. Thanks to \cite[(2.4)]{Fa5}, we thus have
\[ \Omega^{2m}_{L_1T}(\widehat{Z}_r^L(\lambda)|_{L_1T}) \cong \widehat{Z}_r^L(\lambda)|_{L_1T}\!\otimes_k\!k_{pm\alpha},\]
whence
\[ \widehat{Z}_r^L(\mu)|_{L_1T}  \cong \widehat{Z}_r^L(\lambda)|_{L_1T}\!\otimes_k\!k_{pm\alpha}.\]
Since the weights of the former module are bounded above by $\mu$, while those of the latter are $\le \lambda\!+\!pm\alpha$, our assertion follows. \ \ \ \ $\Diamond$

\medskip
(g) \ {\it We have $\dep(\mu)=1$.}

\smallskip
\noindent
In light of (f), we have
\[ \langle \mu+\rho,\alpha^\vee\rangle = \langle \lambda +\rho,\alpha^\vee\rangle + pm\langle \alpha,\alpha^\vee\rangle \equiv \langle \lambda +\rho,\alpha^\vee\rangle \ \
\modd p\ZZ.\]
Since $ \langle \lambda +\rho,\alpha^\vee\rangle \not \in p\ZZ$, it follows that $\dep(\mu) = 1$. \ \ \ \ $\Diamond$

\medskip
(h) \ {\it If $r>1$, then $m=1$.}

\smallskip
\noindent
Let $\dep_L(\lambda)$ denote the depth of $\lambda$, viewed as a weight of $L$. Since $\alpha$ is a simple root, we have
\[ \langle\rho_L,\alpha^\vee \rangle = 1 = \langle\rho,\alpha^\vee \rangle,\]
so that the choice of $\alpha$ implies $\dep_L(\lambda) = \dep(\lambda) = 1$. Thanks to (d), it suffices to verify the result for $L$, so that $\rho = \frac{1}{2}\alpha$. By (a), (b) and (f) it follows that
\[ \lambda -(p^r\!-\!1)\alpha \le \mu-mp^r\alpha = \lambda +mp\alpha -mp^r\alpha,\]
whence
\[ 1 \le p(m-(m\!-\!1)p^{r-1}).\]
As $r\ge 2$, this only holds for $m=1$. \ \ \ \ $\Diamond$

\medskip
(i) \ {\it We have $r=1$.}

\smallskip
\noindent
Suppose that $r \ge 2$. Then (h) implies $m=1$. As before, we will be working with the Levi subgroup $L$, defined by the simple root $\alpha \in \Sigma$. Identifying weights with integers, we may assume that $\lambda \in \{0,\ldots,p^r\!-\!1\}$. Note that $\rho_L$ and $\alpha$ correspond to $1$ and $2$, respectively. As $r\ge 2$, a consecutive application of (a), (d) and (f) implies
\[ \Omega^2_{L_rT}(\widehat{Z}_r^L(\lambda)) \cong \widehat{Z}_r^L(\lambda\!+\!2p).\]
Let $\gamma := 2(p^r\!-\!1)-\lambda-2p$ and write $\gamma = \gamma_0+p^r\gamma_1$, where $\gamma_0 \in \{0,\ldots,p^r\!-\!1\}$. Thanks to \cite[(II.9.6(5))]{Ja3} and \cite[(II.9.7(1))]{Ja3}, we have
\[ \Soc_{L_rT}(\widehat{Z}_r^L(\lambda\!+\!2p)) \cong \widehat{L}_r(\gamma\!-\!2p^r\gamma_1),\]
so that the above leads to a short exact sequence
\[ (\ast\ast) \ \ \ \ \ \ \ \ \ \ (0) \lra \widehat{Z}_r^L(\lambda\!+\!2p) \lra \widehat{P}_r(\gamma\!-\!2p^r\gamma_1) \lra \widehat{P}_r(\lambda) \lra \widehat{Z}_r^L(\lambda) \lra (0)\]
of $L_rT$-modules. Two cases arise:

($i_a$) \ $\lambda+2p \le p^r-2$.

\noindent
Then we have $\gamma_0 = p^r-2-\lambda-2p$ and $\gamma_1=1$, so that $\gamma-2p^r = -(\lambda+2p+2)$. It follows from our sequence ($\ast\ast$) that $\widehat{Z}_r^L(-(\lambda+2p+2))$ is a filtration factor of $\widehat{P}_r(\lambda)$, so that BGG reciprocity \cite[(II.11.4)]{Ja3} implies
\[ -(\lambda+2p+2)> \lambda,\]
relative to the partial ordering given by the positive root, a contradiction.

($i_b$) \ $\lambda+2p \ge p^r-1$.

\noindent
Since $r\ge 2$, we obtain $\gamma_1 = 0$, so that
\[ \Soc_{L_rT}(\widehat{Z}_r^L(\lambda\!+\!2p)) \cong \widehat{L}_r(\gamma).\]
BGG reciprocity yields $\gamma>\lambda$, whence
\[ (\dagger) \ \ \ \ \ \ \ \ \ \ \ \ \lambda < p^r-p-1\]
relative to the ordering of the natural numbers. As $\widehat{Z}_r^L(\lambda\!+\!2p)$ is not projective, we actually have $\lambda+2p \ge p^r$, whence
\[ p^r-2p \le \lambda < p^r-p.\]
Standard $\SL(2)_1$-theory in conjunction with \cite[(1.1)]{HJ} then yields
\[ \dim_k \widehat{P}_r(\lambda) = 4p^r.\]
On the other hand, the inequality ($\dagger$) yields
\[\gamma = 2p^r-2-\lambda-2p > 2p^r-2-2p-p^r+p+1 = p^r-p-1,\]
so that another application of \cite[(1.1)]{HJ} implies
\[ \dim_k \widehat{P}_r(\gamma) = 2p^r.\]
The exact sequence ($\ast\ast$), however, yields $\dim_k\widehat{P}_r(\lambda) = \dim_k\widehat{P}_r(\gamma)$, a contradiction. \ \ \ \ $\Diamond$

\medskip
\noindent
As an upshot of the above, our result holds for weights of depth $\dep(\lambda) = 1$, and we now suppose that $2 \le d+1 = \dep(\lambda) \le r$.

We consider the case, where $G$ is semi-simple and simply connected. Since $\widehat{Z}_r(\mu) \cong \Omega^{2m}_{G_rT}(\widehat{Z}_r(\lambda))$, we may apply Proposition \ref{PH5} to see
that
\[ \dep(\mu) = \ph_\Sigma(Z_r(\mu)) = \ph_\Sigma(Z_r(\lambda)) = \dep(\lambda).\]
According to \cite[(6.2),(6.6)]{FR}, the functor
\[ \Phi : \modd G_rT \lra \modd G_{r-d}T \ \ ; \ \ M \mapsto \Hom_{G_d}(\St_d,M)^{[-d]}\]
sends $\widehat{Z}_r(\lambda)$ and $\widehat{Z}_r(\mu)$ to the modules $\widehat{Z}_{r-d}(\lambda')$ and $\widehat{Z}_{r-d}(\mu')$, defined by weights of depth $1$, respectively.
Moreover, we have
\[ \widehat{Z}_{r-d}(\mu') = \Phi(\widehat{Z}_r(\mu)) \cong \Phi(\Omega^{2m}_{G_rT}(\widehat{Z}_r(\lambda))) \cong \Omega^{2m}_{G_{r-d}T}(\Phi(\widehat{Z}_r(\lambda)))
\cong  \Omega^{2m}_{G_{r-d}T}(\widehat{Z}_{r-d}(\lambda')).\]
The first part of the proof now implies $r\!-\!d =1$ and provides a simple root $\alpha \in \Sigma\setminus \Psi^1_{\lambda'}$ such that $\mu' =\lambda' + mp\alpha$. The identities $\lambda = p^d\lambda' + (p^d\!-\!1)\rho$ and $\mu = p^d\mu' + (p^d\!-\!1)\rho$ thus yield
\[ \mu = p^d\lambda' +mp^{d+1}\alpha + (p^d\!-\!1)\rho = \lambda + mp^r\alpha\]
as well as
\[ \langle\lambda+\rho,\alpha\rangle = p^d\langle\lambda'+\rho,\alpha\rangle \not \in p^{d+1}\ZZ=p^r\ZZ,\]
so that $\alpha \in \Sigma \setminus \Psi^r_\lambda$.

Suppose the result holds for a covering group $\tilde{G}$ of $G$ with maximal torus $\tilde{T}$ such that the canonical morphism sending $\tilde{G}$ to $G$ maps $\tilde{T}$ onto $T$. 
Owing to \cite[(3.5)]{FR}, $\modd G_rT$ is the sum of those blocks of $\modd \tilde{G}_r\tilde{T}$, whose characters belong to $X(T) \subseteq X(\tilde{T})$. Consequently, our result 
then also holds for $G$ and $T$. The proof may now be completed by repeated application of this argument (cf.\ \cite[(6.6)]{FR}).  

(2) In view of (1), an application of \cite[(II.9.2)]{Ja3} gives $\Omega^{2m}_{G_rT}(\widehat{Z}_r(\lambda)) \cong  \widehat{Z}_r(\lambda)\!\otimes_k\!k_{mp^r\alpha}$. Thus,
$Z_r(\lambda)$ is periodic, and $\cx_{G_rT}(\widehat{Z}_r(\lambda)) = \cx_{G_r}(Z_r(\lambda)) = 1$. Hence Theorem \ref{MC1} and Theorem \ref{MC3} provide an element $\beta \in 
\Psi\cup \{0\}$ and $s \in  \{0,\ldots, r\!-\!1\}$ such that

(a) \ $m = \ell p^s$ for some $\ell>0$, and

(b) \ $\Omega^{2p^s}_{G_rT}(\widehat{Z}_r(\lambda)) \cong \widehat{Z}_r(\lambda)\!\otimes_k\!k_{p^r\beta}$.

\noindent
Consequently,
\[ \widehat{Z}_r(\lambda\!+\!\ell p^r\beta) \cong \widehat{Z}_r(\lambda)\!\otimes_k\!k_{\ell p^r\beta} \cong \Omega^{2m}(\widehat{Z}_r(\lambda) \cong 
\widehat{Z}_r(\lambda\!+\!mp^r\alpha),\]
so that $\ell p^r\beta = mp^r\alpha$, whence $\beta = p^s\alpha \ne 0$. Since $\alpha$ and $\beta$ are roots, we conclude that $s=0$ and $\beta = \alpha$. As a result, (b) gives 
$\Omega^2_{G_rT}(\widehat{Z}_r(\lambda)) \cong \widehat{Z}_r(\lambda)\!\otimes_k\!k_{p^r\alpha}$, and our assertion follows. \end{proof}

\bigskip
\noindent
Given $\lambda \in X(T)$, we recall that $P_r(\lambda)$ denotes the projective cover of the simple $G_r$-module $L_r(\lambda)$. If $L_r(\lambda)$ is not projective,  we let $\Ht_r(\lambda) = \Rad(P_r(\lambda))/\Soc(P_r(\lambda))$ be its \emph{heart}. Recall that $\Gamma_s(G_r)$ denotes the stable Auslander-Reiten quiver of the self-injective algebra $kG_r=\Dist(G_r)$.

We record an immediate consequence of Theorem \ref{VM3}, which generalizes and corrects \cite[(4.3)]{Fa5}.

\bigskip

\begin{Cor} \label{VM4} Suppose that $G$ is defined over $\FF_p$ and that $p$ is good for $G$. Let $\lambda \in X(T)$ be a weight of depth $\dep(\lambda)\le r$. If $Z_r(\lambda)$ and $Z_r(\mu)$ belong to the same component of $\Gamma_s(G_r)$, then $Z_r(\mu) \cong Z_r(\lambda)$.\end{Cor}

\begin{proof} According to \cite[(4.4)]{FR}, the baby Verma modules $Z_r(\lambda)$ and $Z_r(\mu)$ are quasi-simple. Since $kG_r$ is symmetric, this implies the existence of $m \in \ZZ\setminus \{0\}$ with
\[ \Omega_{G_r}^{2m}(Z_r(\lambda)) \cong Z_r(\mu).\]
Consequently,
\[ \fF(\widehat{Z}_r(\mu)) \cong Z_r(\mu) \cong \Omega_{G_r}^{2m}(\fF(\widehat{Z}_r(\lambda))) \cong \fF( \Omega_{G_rT}^{2m}(\widehat{Z}_r(\lambda))),\]
and \cite[(4.1)]{GG1} provides $\gamma \in X(T)$ such that
\[ \widehat{Z}_r(\mu\!+\!p^r\gamma) \cong \Omega_{G_rT}^{2m}(\widehat{Z}_r(\lambda)).\]
Theorem \ref{VM3} gives $\mu + p^r\gamma - \lambda \in p^rX(T)$, so that \cite[(II.3.7(9)]{Ja3} implies $Z_r(\mu) \cong Z_r(\mu\!+\!p^r\gamma) \cong Z_r(\lambda)$. \end{proof}

\bigskip
\noindent
The analogue of Corollary \ref{VM4} for $\Gamma_s(G_rT)$ does not hold. The following example falsifies \cite[(4.3(1))]{Fa5}, whose proof is based on an incorrect citation of \cite[(II.11.7)]{Ja3}).

\bigskip

\begin{Example} Let $G = \SL(2)$ and consider a non-projective $\SL(2)_1T$-module $\widehat{Z}_1(\lambda)$. Then $\widehat{Z}_1(\lambda)$ has complexity $\cx_{\SL(2)_1T}(\widehat{Z}_1(\lambda)) = 1$, and
Theorem \ref{MC1} implies
\[ \Omega^2_{\SL(2)_1T}(\widehat{Z}_1(\lambda)) \cong \widehat{Z}_1(\lambda)\!\otimes_k\! k_{p\alpha} \cong \widehat{Z}_1(\lambda\!+\!p\alpha),\]
where $\alpha$ denotes the positive root of $\SL(2)$. As a result, the component of $\Gamma_s(\SL(2)_1T)$ containing $\widehat{Z}_1(\lambda)$ contains infinitely many baby Verma modules. \end{Example}

\bigskip
\noindent
Theorem \ref{VM3} actually shows that $\mu = \lambda + mp^r\alpha$ for every $\alpha \in \Sigma \setminus \Psi^r_\lambda$, so that $\Sigma\setminus \Psi^r_\lambda$ is a singleton. This means that the foregoing example is essentially the only exception.

\bigskip

\begin{Cor} \label{VM5} Suppose that $G$ defined over $\FF_p$ with $p\ge 7$. Let $\lambda \in X(T)$ be a weight of depth $\dep(\lambda) \le r$ 
such that there exist $\mu \in X(T)$ and $m \in \NN$ with $\Omega^{2m}_{G_rT}(\widehat{Z}_r(\lambda)) \cong \widehat{Z}_r(\mu)$. Then the following statements hold:

{\rm (1)} \ $G_r \cong S_r \times H_r$, with $S \cong \SL(2)$ and $H$ being reductive.

{\rm (2)} \ There is a functor $\fG : \modd S_rT' \lra \modd G_rT$ and a weight $\lambda' \in X(T')$ that $\fG$ sends $\widehat{Z}_r(\lambda')$ onto $\widehat{Z}_r(\lambda)$ and 
induces an isomorphism $\Theta_r(\lambda') \cong \Theta_r(\lambda)$ between the AR-components containing $\widehat{Z}_r(\lambda')$ and $\widehat{Z}_r(\lambda)$, respectively. 

{\rm (3)} \ There exists a simple root $\alpha \in \Psi$ such that $\{\widehat{Z}_r(\lambda\!+\!np^r\alpha \ ; \ n \in \ZZ\}$ is the set of those baby Verma modules that belong to 
$\Theta(\lambda)$.\end{Cor}

\begin{proof} (1) Thanks to Theorem \ref{VM3}, we have $\dep(\lambda) = r$ as well as $\mu = \lambda + mp^r\alpha$ for some simple root $\alpha \in \Sigma$. The proof of Proposition \ref{PH5} now yields $\ph_{(U_\alpha)_r}(\widehat{Z}_r(\lambda)) = r$.

In view of \cite[(II.3.7)]{Ja3}, our assumption implies
\[ Z_r(\lambda) \cong Z_r(\mu) \cong \fF(\widehat{Z}_r(\mu)) \cong \fF(\Omega^{2m}_{G_rT}(\widehat{Z}_r(\lambda))) \cong \Omega^{2m}_{G_r}(\fF(\widehat{Z}_r(\lambda)))
\cong \Omega^{2m}_{G_r}(Z_r(\lambda)),\]
so that $Z_r(\lambda)$ is a periodic module. In particular, the module $\widehat{Z}_r(\lambda)$ has complexity $\cx_{G_rT}(\widehat{Z}_r(\lambda)) =1$. Consequently,
$V_r(G)_{\widehat{Z}_r(\lambda)}$ is a one-dimensional, irreducible variety and
\[ V_r(G)_{\widehat{Z}_r(\lambda)} = V_r(U_\alpha)_{\widehat{Z}_r(\lambda)}.\]
In view of Theorem \ref{PH3}, the group $\cU_{\widehat{Z}_r(\lambda)}$ is a subgroup of $(U_\alpha)_r$ of height $r$, whence $\cU_M = (U_\alpha)_r$.

The Borel subgroup $B$ acts on $kG_r$ via the adjoint representation. Since the twist $Z_r(\lambda)^{(b)}$ of a baby Verma module $Z_r(\lambda)$ by $b \in B$ is isomorphic to
$Z_r(\lambda)$, it follows that the variety $V_r(G)_{\widehat{Z}_r(\lambda)}$ is $B$-invariant. This readily implies $B\dact (U_\alpha)_r = B\dact \cU_{\widehat{Z}_r(\lambda)}
\subseteq \cU_{\widehat{Z}_r(\lambda)} = (U_\alpha)_r$, and the arguments of \cite[(7.3)]{FR} yield a decomposition $G = SH$ as a semidirect product, with $S$ being simple of rank 
$1$. It follows that $G_r \cong S_r\times H_r$. 

(2) Let $T = T'T''$ be the corresponding decomposition of the chosen maximal torus $T$ of $G$. The arguments of \cite[(7.3)]{FR} also provide a decomposition
\[ \widehat{Z}_r(\lambda) \cong \widehat{Z}_r(\lambda')\!\otimes_k\!P\]
as an outer tensor product of the $S_rT'$-module $\widehat{Z}_r(\lambda')$ and the simple projective $H_rT''$-module $P$. We now consider
\[ \fG : \modd S_rT' \lra \modd G_rT \ \ ; \ \ M \mapsto M\!\otimes_k\!P.\]
As observed in the proof of Theorem \ref{WT3}, this functor identifies $\modd S_rT'$ with a sum of blocks of $\modd G_rT$ and in particular induces an isomorphism $\Theta(\lambda') 
\cong \Theta(\lambda)$.

(3) Let $\cA := \{\widehat{Z}_r(\lambda\!+\!np^r\alpha) \ ; \ n \in \ZZ\}$. According to Theorem \ref{VM3}(2), $\cA = \{\Omega^{2n}_{G_rT}(\widehat{Z}_r(\lambda) ) \ ; \ n \in 
\ZZ\}$ is contained in $\Theta(\lambda)$. 

Suppose that $\widehat{Z}_r(\nu)$ belongs to $\Theta(\lambda)$. By virtue of \cite[(4.5)]{FR}, the modules $\widehat{Z}_r(\lambda)$ and $\widehat{Z}_r(\nu)$ are quasi-simple. Thanks to Corollary \ref{HA3},  there thus exists $n \in \ZZ$ such that $\Omega_{G_rT}^{2n}(\widehat{Z}_r(\lambda)) \cong \widehat{Z}_r(\nu)$. If $n>0$, then Theorem \ref{VM3} provides a simple root $\beta \in \Psi$ such that

(a) \ $\nu = \lambda+np^r\beta$, and

(b) \ $\widehat{Z}_r(\lambda\!+\!p^r\alpha) \cong \Omega^2_{G_rT}(\widehat{Z}_r(\lambda)) \cong \widehat{Z}_r(\lambda\!+\!p^r\beta)$.

\noindent
Thus, $\beta = \alpha$ and $\widehat{Z}_r(\nu)$ belongs to $\cA$. 

Alternatively, $\Omega_{G_rT}^{-2n}(\widehat{Z}_r(\nu)) \cong \widehat{Z}_r(\lambda)$ and the foregoing arguments yield $\lambda = \nu-np^r\alpha$. \end{proof}

\bigskip

\bigskip

\begin{center}

\bf Acknowledgement

\end{center}
Parts of this paper were written while the author was visiting the Isaac Newton Institute in Cambridge. He would like to take this opportunity to thank the members of the Institute for their
hospitality and support.

\end{document}